\RequirePackage[l2tabu, orthodox]{nag}
\documentclass[11pt]{amsart} %{{{1

\author{Andy Hammerlindl}
\address{School of Mathematical Sciences, Monash University, Victoria $3800$ Australia} \urladdr{ http://users.monash.edu.au/~ahammerl/}  \email{andy.hammerlindl@monash.edu}
\title{Horizontal vector fields and Seifert fiberings}
\usepackage{amsmath}
\usepackage{amssymb}
\usepackage{amsfonts}
\usepackage{amsthm}
\usepackage{amscd}
\usepackage{graphicx}
\usepackage{setspace}
\usepackage{enumitem}
\usepackage{cleveref}
\usepackage{multicol}
\usepackage{fourier}

\usepackage[T1]{fontenc}

\makeatletter
\def\saveenum{\xdef\@savedenum{\the\c@enumi\relax}}
\def\resetenum{\global\c@enumi\@savedenum}
\makeatother

    \newcommand{\bbR}{\mathbb{R}}
    \newcommand{\bbC}{\mathbb{C}}

    \newcommand{\bbZ}{\mathbb{Z}}

    \newcommand{\bbT}{\mathbb{T}}

    \newcommand{\bbS}{\mathbb{S}}

    \newcommand{\UTbbR}{U T \bbR}

    \newcommand{\subof}{\subset}
    \newcommand{\ti}{\times}
    \newcommand{\sans}{\setminus}

    \newcommand{\inv}{^{-1}}

    \newcommand{\Sig}{\Sigma}

    \newcommand{\al}{\alpha}
    
    \newcommand{\bt}{\beta}

    \newcommand{\qandq}{\quad \text{and} \quad}
    \newcommand{\qorq}{\quad \text{or} \quad}

    \newcommand{\id}{\operatorname{id}}
    
    \newcommand{\ior}{\operatorname{int}}

    \newcommand{\del}{\partial}

    \newcommand{\Rab}{R_{a b}}
    \newcommand{\Uab}{U_{a b}}

    \newcommand{\UT}{U T}
    \newcommand{\UTSig}{\UT \Sig}
    \newcommand{\DMSig}{\mathcal{D}(M \to \Sig)}

    \newcommand{\normalize}[1]{\frac{#1}{\|{#1}\|}}

\newcommand{\covercase}[2]{
    If $\Sig$ is the ${#1}$ orbifold, then \\
    $\UTSig$ \ $d$-fold covers \ $\phantom{-}\UTSig$ \ if
    \ $d \equiv +1$ mod ${#2}$, and \\
    $\UTSig$ \ $d$-fold covers \ $-\UTSig$ \ if
    \ $d \equiv -1$ mod ${#2}$.
}

\newcommand{\hvf}{horizontal vector field}
\newcommand{\hvfs}{\hvf s}

\newcommand{\simpleCD}[9]{
\begin{CD}
#1 @>#2>> #3 \\
@VV#4V  @VV#6V \\
#7 @>#8>> #9
\end{CD}
}

\numberwithin{equation}{section}
\newtheorem{thm}[equation]{Theorem}
\newtheorem{cor}[equation]{Corollary}
\newtheorem{lemma}[equation]{Lemma}

\newtheorem{prop}[equation]{Proposition}

\theoremstyle{remark}
\newtheorem*{remark} {\textbf{Remark}}

\begin{document}

\maketitle

\begin{abstract}
    This paper gives a classification of the topology of
    vector fields which are
    nowhere tangent to the fibers of a Seifert fibering.
\end{abstract}
% Intro {{{1

\section{Introduction} \label{sec:intro}

This paper gives a classification of
the topology of horizontal vector fields on
Seifert fiber spaces.
Here, \emph{horizontal} means that the
vector field is nowhere tangent to the Seifert fibering.
This work was inspired in part by the classification of horizontal foliations
on Seifert fiber spaces
\cite{nai1994foliations}.
Depending on the geometry of the 3-manifold,
the condition of having a horizontal vector field can
be either more restrictive or less restrictive
than having a horizontal foliation.
The key observation is that a horizontal
vector field corresponds to a fiber-preserving map from the
Seifert fiber space to the unit tangent bundle of the base orbifold.

We give definitions and review the theory for
orbifolds and Seifert fiberings in sections \ref{sec:orbifold}
and \ref{sec:seifert} respectively,
and develop a general theory for \hvfs{} in \cref{sec:general}.
In this introductory section, however,
%This map is explained in detail in sec:general, but
we first list
out the possibilities based on the geometry of the base orbifold.

First and most interesting are Seifert fiber spaces
where the base orbifold is a bad orbifold.
These are lens spaces and each lens space $L(p,q)$ supports
infinitely many distinct Seifert fiberings.
%The exact nature of which lens spaces and which fiberings
%admit hvfs() is surprising.

\begin{thm} \label{thm:lens}
    Suppose $M$ is diffeomorphic to the lens space $L(p,q)$ with $p \ge 0.$
    \begin{enumerate}
        %If p = 0, then no Seifert fibering on M
        %has a hvf().
        %---
        \item        %If 1 <= p <= 2,
        If $p = 1$ or $p = 2$,
        then every Seifert fibering on
        $M$ has a \hvf{}.
        \item
        If $p \ge 3$ and $q \equiv \pm 1 \bmod p$,
        then $M$ has infinitely many Seifert fiberings
        which support \hvfs{} and infinitely many which do not.
        \item
        If $p \ge 8, p$ is divisible by $4,$
        and $q \equiv \tfrac{1}{2}p \pm 1 \bmod p$,
        then $M$ has exactly one Seifert fibering
        which supports a \hvf{} and all others do not.
        \item
        For all other cases of $p$ and $q,$
        no Seifert fibering on $M$ has a \hvf{}.
    \end{enumerate} \end{thm}    
These cases are explained in detail in \cref{sec:lens}.
To prove \cref{thm:lens}, we introduce the notion of a ``marked lens space''
which carries additional information on the orientation of the two solid
tori which are glued together to construct the lens space.
Under such a notion, $L(p, q)$ and $L(-p, -q)$ are distinct marked lens spaces
even though they are diffeomorphic. Further, except for a few special
cases,
every Seifert fibering of $L(-p, -1)$ has a \hvf{} while no Seifert fibering
of $L(p, 1)$ has a \hvf{}.

We next consider Seifert fiberings where the base orbifold is elliptic;
that is, the orbifold is finitely covered by the 2-sphere.
Some of these spaces, such as the 3-sphere, are also lens spaces and are
handled by the previous theorem.

\begin{thm} \label{thm:elliptic}
    Suppose $M$ is a Seifert fiber space with elliptic base orbifold
    and $M$ is not diffeomorphic to a lens space.
    Then $M$ has a \hvf{} if and only if
    $M,$ up to orientation, is
    the unit tangent bundle of the base orbifold.
\end{thm}
Next consider Seifert fiber spaces over parabolic orbifolds.
Such an orbifold is covered by the Euclidean plane and corresponds to a wallpaper
group with only rotational symmetry.

\begin{thm} \label{thm:parabolic}
    Suppose $M$ is a Seifert fiber space with parabolic base orbifold.
    Then $M$ has a \hvf{} if and only if either
    \begin{enumerate}
        \item $M,$ up to orientation, is
        the unit tangent bundle of the base orbifold, or
        \item
        the base orbifold is a surface (either the 2-torus or the Klein
        bottle).
    \end{enumerate} \end{thm}
\Cref{sec:para} explains both of the above cases in detail.
Note that the two cases overlap when the base orbifold is a surface.

Finally, the most general case is for hyperbolic orbifolds.

\begin{thm} \label{thm:hyperfull}
    Suppose $M$ is a Seifert fiber space with hyperbolic base orbifold.
    Then the following are equivalent{:}
    \begin{enumerate}
        \item $M$ has a \hvf{},
        \item
        $M$ finitely covers the unit tangent bundle of the base orbifold,
        \item
        $M$ supports an Anosov flow.
    \end{enumerate} \end{thm}
The equivalence $(2)$ $\Leftrightarrow$ $(3)$ in \cref{thm:hyperfull}
was proved by Barbot, extending a result of Ghys on circle bundles
\cite{ghy1984flots,bar1996flots}.
The proof of $(1)$ $\Leftrightarrow$ $(2)$ is discussed in \cref{sec:hyp}.
\Cref{thm:hyperfull} was part of the original motivation
for exploring the properties of these vector fields.
In joint work with Mario Shannon and Rafael Potrie,
we analyze partially hyperbolic dynamical systems
on Seifert fiber spaces
\cite{hps2018seifert}.
Most of that paper deals only with the case of dynamical systems
defined on circle bundles.
However, \cite[\S $5$]{hps2018seifert} employs \cref{thm:hyperfull}
above to extend these results to the case of a general Seifert fiber
space. See the cited paper for further details.

\Cref{sec:homotopy} studies the homotopy classes of \hvfs{},
and
\cref{sec:boundary} looks at the case of Seifert fiberings on manifolds with
boundary.

\medskip

\section{Orbifolds} \label{sec:orbifold} %{{{1

Intuitively,
an orbifold is an object that locally resembles
a quotient of $\bbR^n$ by a finite group.
In this paper, we will only consider those orbifolds
that appear as the base space of a Seifert fibering,
and so we give a simpler definition restricted to this case.

Let $D^2 = \{ z \in \bbC : |z| \le 1 \}$ be the closed unit disk.
For an integer $a \ge 1,$
define $r_a : D^2 \to D^2, z \to e^{2 \pi i/a} z.$
This is a rotation of order $a$.
Define
$D_a$ as the topological quotient space
$D_a = D^2 / r_a.$
If $a > 1,$ then $D_a$ is homeomorphic to a disk,
but is not a smooth surface.
By a slight abuse of notation,
we write $0 \in D_a$ for the image of $0 \in D^2$ under the quotient.
If $U$ is an open subset of $D_a,$
then $U$ is a smooth manifold if and only if $0 \notin U.$
Let $\ior(D_a)$ denote the interior of $D_a.$

An \emph{orbifold} $\Sig$ is a closed topological surface
equipped with an atlas of charts such that
\begin{enumerate}
    \item each chart is of the form
    \begin{math}
        (\varphi, U, a)
    \end{math}
    where $a \ge 1$ is an integer,
    $U \subof \Sig$ is open, and
    $\varphi : U \to \ior(D_a)$ is a homeomorphism;
    \item
    the charts cover $\Sig;$ and
    \item
    if $(\varphi, U, a)$ and $(\psi, V, \hat a)$
    are distinct charts in the atlas,
    then $\varphi(U \cap V)$ and $\psi(U \cap V)$
    are smooth manifolds and
    \begin{math}
        \psi \circ \varphi \inv : \varphi(U \cap V) \to \psi(U \cap V)
    \end{math}
    is a diffeomorphism.
\end{enumerate}
The final item above implies
in particular
that if $a > 1$ then $\varphi \inv(0) \notin V$ and
if $\hat a > 1$ then $\psi \inv(0) \notin U.$
We call $x \in \Sig$ a \emph{cone point} if the atlas has a chart
$(\varphi, U, a)$ such that $\varphi(x) = 0$ and $a > 1.$
Here, $a$ is the \emph{order} of the cone point.
In this version of the definition of an orbifold,
each cone point is contained in exactly one chart.
Since the orbifold is compact, it can have at most finitely many
cone points.

In this paper,
we refer to orbifolds using the orbifold notation of Thurston and Conley
\cite{con1992orbifold}.
In this notation, the starting surface is a sphere
and orbifolds are constructed from this by surgery.
A positive integer $a$ in the notation corresponds to excising a disk
and gluing in $D_a$ in its place.
For example, the $237$ orbifold is a sphere with cone points
of orders $2, 3,$ and $7$ added.
We allow ones in the notation in order to make some results easier to state.
For instance, consider the family of $pp$ orbifolds with $p \ge 1.$
This is an infinite family of orbifolds that contains the $11$ orbifold,
i.e, the 2-sphere, as a member of the family.
The symbol ``$\ti$'' corresponds to replacing a disk with a cross cap,
and so
the $22\ti$ orbifold is the sphere with two cones points of order two
and one cross-cap added.
The symbol ``$o$'' corresponds to adding a handle, and so
the $23o o$ orbifold is
a genus-two surface with points of orders $2$ and $3$ added.
We do not consider orbifolds with silvered edges or corner reflectors
and so the ``*'' symbol will not occur in any of the notation here.

For an orbifold $\Sig,$ the \emph{underlying topological surface} $\Sig_0$
is just the orbifold itself treated as a topological surface
and forgetting any of the additional information given by the charts
$(\varphi, U, a).$

\medskip{}

For a smooth surface $\Sig$ equipped with a Riemannian metric,
one can define the unit tangent bundle $\UTSig$
consisting of all tangent vectors of length $ \| v \| = 1.$
This is a 3-manifold $\UTSig$
with a natural circle fibering coming from the projection
$\UTSig \to \Sig.$
%Note that if Sig is equipped with two different metrics,
%the resulting UTSig will still be diffeomorphic to each other.
We now extend this notion of a unit tangent bundle
to the case of an orbifold.

First, consider a choice of Riemannian metric $g$ on the disk $D^2$
such that with respect to this metric the rotation $r_a : D^2 \to D^2$
is an isometry.
If $U$ is an open subset of $D_a$ with $0 \notin U,$
then $g$ descends to a Riemannian metric on $U.$
By a slight abuse of notation, we call such a $g$ a
Riemannian metric on $D_a.$
Define $T D_a$ as the quotient of $T D^2$ by the
tangent map $T r_a : T D^2 \to T D^2.$
The tangent bundle $T D^2$
has the structure of a smooth 4-manifold with boundary,
but if $a > 1,$ then the quotient $T D_a$ does not
have such structure.
Since $r_a$ is an isometry,
its tangent map restricts to a diffeomorphism
of the unit tangent bundle of $D^2.$
This diffeomorphism and its iterates
$T r_a^i|_{U T D^2}$ for $1 \le i < a$
have no fixed points.
From this property, one can see that the quotient
$U T D_a$ = $U T D^2 / T r_a$
is a smooth 3-manifold with boundary.
In fact, it is a solid torus.

For an orbifold $\Sig,$
we define the tangent bundle using using equivalence classes of tangent
vectors in charts.
%That is, if phi : U -> bbD_{a_U} and psi : V -> bbD_{a_V}
%are overlapping charts, then
That is, consider the set $X$ of all tuples
$(p, (\phi, U, a), v)$
where $p$ is a point in $\Sig,$ $(\phi, U, a)$ is a chart containing $p,$
and $v$ a vector in $T D_a$ based at $\phi(p).$
The equivalence relation is defined by
\[
    (p, (\phi, U, a), v) \sim (q, (\psi, V, \hat a), w)
\]
if and only if
\[
    p = q
    \qandq %T_{phi(p)} (psi circ phi inv)(v) = w.
    T(\psi \circ \phi \inv)(v) = w.
\]
The \emph{tangent bundle} $T \Sig$ of the orbifold
is defined as the set of equivalence classes of $X.$

A \emph{Riemannian metric} on an orbifold $\Sig$
is a choice for every chart $(\phi, U, a)$
of Riemannian metric on $D_a$
such that every transition map
$\psi \circ \phi \inv : \phi(U \cap V) \to \psi(U \cap V)$
is an isometry.
Suppose $\Sig$ is an orbifold equipped with a Riemannian metric.
As the transition maps are isometries,
every element $v \in T \Sig$ has a well-defined length,
even those vectors based at cone points,
and we may define the \emph{unit tangent bundle}
\begin{math}
    \UTSig = \{ v \in T \Sig : \| v \| = 1 \}.
\end{math}

A key observation is that even though $\Sig$ is an orbifold,
its unit tangent bundle is a smooth 3-dimensional manifold.
To see this,
note that $\UTSig$ is covered by open sets of the form
$U T V$ where $V \subof \Sig$ is given by a chart $(V, \phi, a)$
and $T \phi : U T V \to U T D_a$
gives a smooth structure to $U T V.$
Distinct choices of metric on $\Sig$ will yield
unit tangent bundles which are distinct as subsets of $T \Sig,$
but which are nonetheless diffeomorphic to each other.
Because of this,
in what follows we will at times discuss ``the''
unit tangent bundle of an orbifold without giving a specific choice of metric.

The unit tangent bundle of an orbifold is oriented
if even the orbifold itself is not.
This is because at a point $x \in \Sig,$
a choice of orientation of the tangent space $T_x \Sig$
induces an orientation on the circle of unit tangent vectors.
Regardless of the choice of orientation on $T_x \Sig,$
the resulting orientation on the 3-manifold is the same.

\medskip{}

One way to produce an orbifold is to quotient a surface by a discrete group of
isometries.
We call an orbifold $\Sig$ a \emph{good orbifold}
if there is a Riemannian surface $S$ of constant curvature
and a finite group of isometries $G$ acting on $S$
such that $\Sig = S / G$ with the orbifold structure of $\Sig$ 
coming from the smooth structure of $S.$
We do not make a distinction between ``good'' and ``very good''
orbifolds as the notions are equivalent in the 2-dimensional setting.
An orbifold is \emph{bad} if it is not good.
Proofs of the following two results are given
in \cite[Chapter $5$]{cho2012geometric}.
See also the discussion in \cite[\S $2$]{sco1983geometries}.

\begin{prop}
    An orbifold $\Sig$ is bad if and only if
    it is either
    \begin{enumerate}
        \item a sphere with a single cone point of order $p > 1$ added, or
        \item
        a sphere with two cone points of orders $p \ne q$ added.
    \end{enumerate} \end{prop}    
For an orbifold $\Sig$ with
underlying topological surface $\Sig_0$
and cone points of order $\al_1, \ldots, \al_n,$
define the \emph{Euler characteristic} of $\Sig$
as
\[
    \chi(\Sig) = \chi(\Sig_0) - \sum_i \left( 1 - 1/\al_i \right).
\]
One way to think of this definition is that it is the
formula $\chi(\Sig) = V - E + F$
for vertices, edges, and faces, but with a cone point counting
as a fraction $1/\al_i$ of a vertex.

\begin{prop}
    If $\Sig = S / G$ is a good orbifold, then
    $\chi(\Sig) = \frac{1}{|G|} \chi(S).$
\end{prop}    %where |G| is the order of the group G.

We call a good orbifold $\Sig,$ \emph{elliptic} if $\chi(\Sig) > 0,$
\emph{parabolic} if $\chi(\Sig) = 0,$
and
\emph{hyperbolic} if $\chi(\Sig) < 0.$
These cases correspond to the orbifold being covered
by the 2-sphere, 2-torus, or a hyperbolic higher-genus surface
respectively.

\medskip{}

We may define vector fields on orbifolds analogously to vector fields
on manifolds.
If $\Sig$ is an orbifold, a \emph{vector field} on $\Sig$
is a continuous function $v : \Sig \to T \Sig$ such that $\pi(v(x)) = x$
for all $x \in \Sig$ where $\pi$ is the canonical projection $\pi : T \Sig \to \Sig.$

\begin{prop} \label{prop:orbvfield}
    If $v$ is a vector field on an orbifold $\Sig,$
    then $v$ is zero at every cone point on $\Sig.$
    In particular, if $v$ is a non-zero vector field,
    then $\Sig$ is a surface and is either the 2-torus or Klein bottle.
\end{prop}
\begin{proof}
    First, consider the case of a vector field defined on $D_a$ with $a > 1.$
    Such a vector field lifts to a vector field $v : D^2 \to T D^2$
    which is invariant under the rotation $r_a;$
    that is, $T r_a(v(x)) = v(r_a(x))$ for all $x \in D^2.$
    If $x = r_a(x)$ is the center of the disk, this is only possible
    if the vector is zero at this point.
    This shows that a vector field must be zero at the cone point in $D_a.$
    Since orbifolds are locally modelled on $D_a,$ any vector field
    on an orbifold must be zero at all cone points.

    If an orbifold $\Sig$ has a non-zero vector field, then $\Sig$ must have
    no cone points and is a true surface. We are only considering
    the case of surfaces without boundary and the Poincar\'e-Hopf theorem
    implies that $\chi(\Sig) = 0.$
\end{proof}
\section{Seifert fiberings} \label{sec:seifert} %{{{1

We now discuss the definition and properties of Seifert fiberings.

Recall that $D^2$ is a closed unit disk and $r_a$ is a rotation of order $a$
on $D^2.$
Let $S^1 = \del D^2$ and note that $r_a$ restricts to $S^1.$
For coprime integers $a,b$ with $a \ge 1,$
define a diffeomorphism $\Rab$ of the solid torus $D^2 \ti S^1$
by
\[
    \Rab : D^2 \ti S^1 \to D^2 \ti S^1, \quad (z,w) \mapsto (r_a^b(z), r_a(w)).
\]
This generates a cyclic group of order $a$ acting freely on $D^2 \ti S^1$
and the quotient
\[
    \Uab := D^2 \ti S^1 / \Rab
\]
is again a solid torus.
Consider the projection $\Pi : D^2 \ti S^1 \to D^2$ onto the first coordinate.
The fibers of this projection are circles and this fibering is invariant
under $\Rab.$ Hence it defines a fibering of $\Uab$ by circles.
The solid torus $\Uab$ along with this fibering by circles is called
is called a \emph{standard fibered torus}.
The fibering may also be viewed directly as the fibers of a map
from $\Uab$ to $D_a = D / r_a.$

In the special case that $a = 1,$
$\Rab$ is the identity map
and $\Uab = D^2 \ti S^1$ has a trivial fibering. We call this a
\emph{trivial fibered torus}.

\medskip{}

Let $M$ be a closed 3-manifold along with a decomposition of $M$ into a disjoint
union of circles.
We say that $M$ is a \emph{Seifert fiber space}
if for every $x \in M$ there is a solid torus $U$ embedded in $M$
with $x$ in its interior such that $U$ is a union of circles
and these circles give $U$ the structure of a standard fibered torus.
That is, there are integers $a$ and $b$ and a diffeomorphism
from $\Uab$ to $U$ which takes circles to circles.

Let $\Sig$ be the topological space defined by quotienting $M$ along these
circles.
Then any standard fibered torus $U$ embedded in $M$
quotients down to a subset of $\Sig$
with the same topological and smooth structure
as the quotiented disk $D_a.$
From this, one may show that the Seifert fibering on $M$
induces the structure of an orbifold on $\Sig.$
The projection $\pi : M \to \Sig$ determines the fibering of circles on $M,$
and so we will often refer to a \emph{Seifert fibering} as given
by a map $M \to \Sig$
where $M$ is a closed 3-manifold and $\Sig$ is an orbifold.
If $x$ is a cone point in $\Sig,$ then the fiber $\pi \inv(x)$ is called
an \emph{exceptional fiber}.
If $x$ is not a cone point, then $\pi \inv(x)$ is called a \emph{regular fiber}.

\medskip{}

We now outline the standard method of construction of a Seifert fiber space
by gluing solid tori into a trivial circle bundle with boundary.
This largely follows the exposition in \cite{jn1983lectures}.
Let $g \ge 0$ be an integer and let $(\al_1, \bt_1), \ldots, (\al_n, \bt_n)$
be such that for each $i, \al_i$ and $\bt_i$ are coprime integers and $\al_i \ge 1.$
Let $F$ be a closed surface of genus $g$ and remove $n$ disjoint open disks from $F$
to produce
$F_0 = F \sans (D_1^2 \cup \ldots \cup D_n^2).$
Let $M_0 = F_0 \ti S^1$ be the trivial circle bundle over $F_0$
and write
\[
    \del M_0 = (S_1^1 \ti S^1) \cup \ldots \cup (S_n^1 \ti S^1).
\]
Define
\begin{align*}
    R
        &= F_0 \ti \{1\},
    \\
    H_i
        &= \{1\} \ti S^1 \subof S^1_i \ti S^1, \qandq
    \\
    C_i
        &= R \cap (S^1_i \ti S^1) = S^1_i \ti \{1\}.
\end{align*}

Here,
$H_i$ is a vertical fiber on the boundary of $M_0$
and $C_i$ is a horizontal circle around the excised disk.
We assume $C_i$ has the same orientation as the boundary of the excised disk.
This means that the boundary components of $R$ have the opposite orientation as
the circles $C_i.$

Let $T_i = D^2 \ti S^1$ be a solid torus.
It has a meridian $M_i = S^1 \ti \{1\} \subof \del T_i$ which bounds a disk inside
$T_i$ and a longitude $L_i = \{1\} \ti S^1 \subof \del T_i.$
Note that $M_i$ and $L_i$ form a basis for the first homology group
of $\del T_i$
and that $C_i$ and $H_i$ form a basis for the first homology group of
$S^1_i \ti S^1 \subof \del M_0.$
At times, we consider rational linear combinations of elements of this
homology,
and so we use homology with rational coefficients throughout.

To glue $T_i$ into $M_0,$ we choose a linear map identifying the two
tori and such that $M_i \sim \al_i C_i + \bt_i H_i$ in homology.
That is, we choose integers $\al'_i$ and $\bt'_i$ such that the matrix
\[
    \begin{pmatrix}
        \al_i & \bt_i \\
        \al'_i & \bt'_i
    \end{pmatrix}  \]
lies in $SL(2, \bbZ)$ and use this as the gluing map.
By a slight abuse of notation, we write
\[
    \begin{pmatrix}
        M_i \\
        L_i \end{pmatrix}
    =
    \begin{pmatrix}
        \al_i & \bt_i \\
        \al'_i & \bt'_i \end{pmatrix}
    \begin{pmatrix}
        C_i \\
        H_i
    \end{pmatrix}  \]
to concisely express how the circles are related in homology.
Once $T_i$ is glued into the manifold, there is a unique way to extend
the fibering to $T_i$ in such a way that $T_i$ is a standard fibered torus.

Note that the longitude of the solid torus is not well defined
and we could replace $L_i$ with a curve homologous to $L_i + k M_i$
for any integer $k.$
The gluing matrix is not unique
and we could replace $\al'_i$ and $\bt'_i$ by $\al'_i + k \al_i$
and $\bt'_i + k \bt_i.$
These, in fact, correspond to the same ambiguity
and they have no impact on the topology of the resulting manifold.

After gluing in all of the tori, the result is
a closed oriented 3-manifold.
We denote this manifold equipped with this Seifert fibering as
\[
    M(g; (\al_1, \bt_1), \ \ldots, \ (\al_n, \bt_n)).
\]
%medskip()

We now consider the case where $g < 0,$
which denotes a Seifert fibering over a non-orientable orbifold.
Readers only interested in Seifert fiberings over
oriented base orbifolds may safely skip over this paragraph.
Let $g < 0$ and let $F$ be a closed non-orientable surface of genus $|g|.$
Let $F_0 = F \sans (D_1^2 \cup \ldots \cup D_n^2).$
Let $G_0$ be the orientable double cover of $F_0$
and let $\tau : G_0 \to G_0$ be an involution such that
$F_0 = G_0 / \tau.$
Define an involution $\phi$ on $G_0 \ti S^1$ by
$\phi(g, z) = (\tau(g), \bar z)$ where complex conjugation $z \mapsto \bar z$
is an involution of the circle.
Then define a twisted circle bundle over $F_0$ by $M_0 = (G_0 \ti S^1) / \phi.$
Define $R$ as the image of $G_0 \ti \{1\}$ under this quotient.
Note that while $M_0$ is orientable,
the fibers are neither orientable nor transversely orientable
on all of $M_0.$
However, we can find a subset $M_1$ of $M_0$ which is a union of
fibers, contains the boundary of $M_0,$ and on which the fibers are
both orientable and transversely orientable.
This subset $M_1$ may be identified with $F_1 \ti S^1$
where $F_1$ is an orientable subset of $F_0$
and such that
\begin{math}
    R \cap M_1 = F_1 \ti \{1\}.
\end{math}
Under this identification, the construction of the Seifert fibering
by gluing proceeds exactly as before.
Because of the existence of this subset $M_1,$
we also freely assume in the proofs in the remainder of this section
that the fibering is always orientable.

\medskip{}

Two Seifert fiberings $M_1 \to \Sig_1$ and $M_2 \to \Sig_2$ are \emph{isomorphic}
if there is a diffeomorphism from $M_1$ to $M_2$ which takes fibers to fibers.
Every Seifert fibering $M \to \Sig$ of an oriented 3-manifold 
is isomorphic to one constructed by gluing in tori as above,
and so we define the \emph{Seifert invariant}
of the fibering as
\[
    (g; (\al_1, \bt_1), \ \ldots, \ (\al_n, \bt_n)).
\]
Two invariants define the same fibering if and only if
one can be transformed into the other by changes of the following form:
\begin{enumerate}
    \item altering each 
    $\bt_i/\al_i$
    by an integer, but keeping $\sum \bt_i/\al_i$ fixed,
    \item
    re-ordering the pairs $(\al_i,\bt_i)$, and
    \item
    inserting or removing pairs of the form $(1,0)$.
\end{enumerate}
For proofs of the above assertions,
see \cite[\S1]{jn1983lectures} or \cite[\S2]{hat2007notes}.

The \emph{Euler number} of a fibering is defined as
\[
    e(M \to \Sig) = - \sum_{i=0}^n \bt_i/\al_i.
\]
One can see that none of the above transformations to the Seifert invariant
will affect the Euler number.

It is possible to extend the notation to define Seifert invariants
for fiberings on non-orientable 3-manifolds.
However, this complicates the notation and is not useful for the current
paper, so we only consider invariants for oriented manifolds.

For most Seifert fiber spaces, the topology of the 3-manifold
uniquely determines the fibering up to isomorphism.
However, there are a few important exceptions, such as for lens spaces.
We discuss this non-uniqueness in detail in \cref{sec:lens}.

\medskip{}

Recall that the unit tangent bundle $\UTSig$ of an orbifold $\Sig$
is a closed 3-manifold.
Further,
the unit tangent bundle of a quotiented disk $D_a = D^2 / r_a$
is a standard fibered torus where the fibers are given
by the projection $T D_a \to D_a.$
From this, one can see that the projection $\UTSig \to \Sig$
gives $\UTSig$ a canonical Seifert fibering.

\begin{prop} \label{prop:UTinvt}
    For an orbifold $\Sig$ with
    underlying topological surface $\Sig_0$ of genus $g$
    and cone points of order $\al_1, \ldots, \al_n,$
    the unit tangent bundle has Seifert invariant
    \[
        (g; \, (1, n - \chi(\Sig_0)), \, (\al_1, -1), \ldots, (\al_n, -1)).
    \] \end{prop}
This fact is stated several times in the literature.
See, for instance, \cite[\S $5$]{ehn1981transverse}.
However, we know of no place that gives a complete proof
including the case of bad orbifolds,
and so
we give a full proof here.

\begin{proof}
    We will first consider the sphere and then certain quotients
    of the sphere before handling general surfaces and orbifolds.

    For the unit tangent bundle of $S^2,$
    travelling once around a circle in the fibering corresponds to
    staying in one place on the globe,
    starting facing a certain direction,
    and then rotating once around
    counter clockwise (CCW).
    Consider a unit vector field on the 2-sphere defined everywhere
    except on small disks centered at the north and south poles.
    This unit vector field always points north
    and plays the role of $R$ in the construction of a Seifert fibering above.

    Looking down on the north pole in a small chart,
    we may trivialize the unit tangent bundle
    in this chart.
    For instance, we may take a section of the unit tangent bundle
    to be a vector field
    which always points ``up'' inside this chart.
    If we consider a small disk around the north pole,
    then its unit tangent bundle is a solid torus
    and the meridian $M_i$ consists of vectors pointing ``up''
    everywhere on the circle which bounds the disk.
    If we walk CCW around this circle once,
    then a north-pointing compass needle will make a complete CCW rotation.
    From this, one can see that $C_i$ is homologous to $M_i + H_i.$
    Similarly at the south pole, if we walk CCW once
    around the circle, then the compass needle will rotate once CCW.
    and so $C_i \sim M_i + H_i$ for this disk as well.
    Together, these show that the Seifert invariant of $\UT S^2$
    is
    \begin{math}
        (0; (1, -1), (1, -1)).
    \end{math}
    \smallskip{}

    Now let $\al > 1$ be an integer and the consider
    a rotation of order $\al$ of the sphere such that the
    quotient by this rotation is an orbifold $\Sig$
    and each of the north and south
    poles of the sphere quotients down to a cone point of order $\al.$
    This quotient defines a map $\pi : U T S^2 \to U T \Sig$
    and one can verify that this a smooth covering map
    between manifolds.
    In what follows, we use uppercase letters
    $R, C_i, H_i, M_i, L_i$
    for curves
    and surfaces in $U T S^2$ and lowercase letters
    $r, c_i, h_i, m_i, \ell_i$
    for corresponding objects in $U T \Sig.$

    The north-pointing vector field is invariant under the rotation,
    and so it quotients down a vector field defined
    everywhere on the orbifold except for small disks
    about each of the cone point.
    The section $R$ defined in the previous case
    covers a section $r$
    defined everywhere except on disks about the two cone points
    and the covering $R \to r$ is of degree $\al.$

    Let $c_i$ denote a boundary circle of $r,$
    but with opposite orientation.
    It is $\al$-fold covered by $C_i.$
    Let $h_i$ denote a regular fiber which is the image of $H_i$
    under the quotient.
    The north pole's meridional disk on $\UT S^2$ maps down injectively
    to a meridional disk on $\UTSig$, and
    the boundary $m_i$ of this disk
    wraps $\al$ times around the cone point in $\Sig.$
    Then
    \[    
        \pi(C_i) = \al c_i, \quad \pi(H_i) = h_i, \qandq \pi(M_i) = m_i
    \]
    in homology.
    It follows from $C_i \sim H_i + M_i$ that $\al c_i$ is homologous to $h_i + m_i.$
    The exact same relations hold at the south pole.
    Hence, $m_i \sim \al c_i - h_i$ at both poles
    and the Seifert invariant for $\UTSig$ is
    \begin{math}
        (0; (\al, -1), (\al, -1)).
    \end{math}
    \smallskip{}

    Now let $\Sig$ be a surface of any genus.
    Put a vector field on $\Sig$ which is non-zero
    everywhere except a finite set consisting
    of sinks, sources, and saddles.
    We can define small disks about each
    of the critical points and each disk yields a solid torus
    in the unit tangent bundle.
    As with the case of the sphere,
    each of the sinks and sources has a gluing
    with $M_i \sim C_i + H_i.$
    A ``compass'' following the vector field around
    a saddle will rotate once clockwise, and so
    one can show that a saddle corresponds to a gluing
    with $M_i \sim C_i - H_i$
    Hence, the circle bundle has an Euler number
    which is the number of sinks and sources minus the number of saddles.
    The Poincar\'e-Hopf theorem tells us that this is exactly
    the Euler characteristic of the surface.

    \smallskip{}

    Finally, consider the general case of an orbifold $\Sig.$
    Consider a vector field on the underlying surface $\Sig_0$ 
    which has finitely many sinks, sources, and saddles.
    Further assume that every cone point in $\Sig$
    corresponds to a point in $\Sig_0$ which is a sink or source for the flow.
    Such a vector field may be produced, for instance,
    by considering a gradient flow where all of the cone points 
    correspond to local minima or maxima.
    The previous case tells us that $\UT \Sig_0$
    has a Seifert invariant which may be written as
    \[    
        (g; \  (1, n - \chi(\Sig_0)), \ 
        \underbrace{(1, -1), \, \ldots, \, (1, -1)}_{n \text{ times}}).
    \]
    Based on our above analysis of the 2-sphere and its quotient,
    one sees that
    the effect of removing a point in a surface
    and replacing it by a cone point of order $\al_i$
    is to remove a pair $(1, -1)$ and replace it
    by $(\al_i, -1)$ in the Seifert invariant.
    This then proves the desired result.
\end{proof}
%For an orbifold Sig with
%underlying topological surface Sig_0
%and cone points of order al_1, ..., al_n,
%define the emph(Euler characteristic) of Sig
%as
%
%    chi(Sig) = chi(Sig_0) - sum_i left( 1 - 1/al_i right).
%
\begin{cor} \label{cor:evchi}
    The equality $e(\UTSig \to \Sig) = \chi(\Sig)$
    holds for any orbifold.
\end{cor}%It also follows that UTSig is an oriented 3-manifold,
%even if Sig contains cross caps.

\begin{proof}
    This follows directly from \cref{prop:UTinvt}
    and the definitions of Euler number and Euler characteristic.
\end{proof}

Suppose $M_1 \to \Sig$ and $M_2 \to \Sig$
are two Seifert fiberings
over the same orbifold.
We call a covering map $u:M_1 \to M_2$ a \emph{fiberwise covering}
if it quotients down to the identity map on $\Sig$.
That is, the following diagram commutes:
\[
    \simpleCD{M_1}{u}{M_2}{}{}{}{\Sig}{\id}{\Sig.}
\]
%In other words, the ``covering'' is happening
%purely in the fibers and not in the orbifold.

As described in \cite[\S2]{jn1983lectures},
a Seifert fibering $M \to \Sig$
may be viewed as the action of $S^1$ on the manifold
where the orbits of the action are fibers of the fibering.
One may also think of the fibers as being equipped
with a metric such that every regular fiber has length exactly one.
The metric varies continuously and so an exceptional
fiber will then have a length of $1/\al_i.$
The $S^1$ action then corresponds to moving along the fibers at constant speed.
A discrete subgroup of $S^1$ will be a cyclic group
$\bbZ / d$ and we may consider the quotient of $M$ by $\bbZ / d.$

\begin{prop} \label{prop:poscoverd}
    Suppose $M$ is a Seifert fiber space with invariant
    \[
        (g; (\al_1, \bt_1), \ldots, (\al_n, \bt_n))
    \]
    and $d$ is a positive integer coprime to $\al_i$
    for all $i.$
    Then the quotient $M / (\bbZ / d)$ is a Seifert fiber space
    with invariant
    \[
        (g; (\al_1, d \bt_1), \ldots, (\al_n, d \bt_n)).
    \]
    Moreover, any fiberwise covering of degree $d \ge 1$ 
    between oriented
    Seifert fiber spaces is of this form.
\end{prop}
\begin{remark}
    A very similar result is stated as
    \cite[Proposition 2.5]{jn1983lectures},
    but without any distinction between the topological and smooth
    categories of manifolds.
    There, they allow the case where $d$ and $\al_i$
    share a common factor.
    In this case, $\bbZ / d$ does not act freely on $M.$
    Further, the quotient destroys the smooth structure of $M$
    and so $M / (\bbZ / d)$ is not a smooth manifold.
    As we are only concerned with smooth fiberwise coverings in the current
    paper, we do not allow such quotients.
\end{remark}
To prove \cref{prop:poscoverd},
we first consider the effect of such quotienting
on a standard fibered torus.
For an integer $d > 1,$ define
\[
    Q_d : D^2 \ti S^1 \to D^2 \ti S^1, \quad (z,w) \mapsto (z, r_d(w))
\]
and note that this map commutes with $\Rab.$

\begin{lemma} \label{lemma:quotienttorus}
    Consider a standard fibered torus $\Uab = D^2 \ti S^1 / \Rab$ with $a > 1$
    and the map $Q_d$ for some $d > 1.$
    \begin{enumerate}
        \item If $a$ and $d$ are coprime,
        then the group $\langle \Rab, Q_d \rangle$
        acts freely on $D^2 \ti S^1$ and the quotient
        is a standard fibered torus.
        \item
        If $a$ and $d$ are not coprime,
        then the group $\langle \Rab, Q_d \rangle$
        does not act freely on $D^2 \ti S^1$ and the quotient
        does not have the structure of
        a smooth manifold with boundary.
    \end{enumerate} \end{lemma}
\begin{proof}
    First, suppose $a$ and $d$ are coprime.
    %Since $a$ and $b$ are also coprime,
    Then, there are integers $n$ and $m$ such that
    \begin{math}
        n d + m a = 1.
    \end{math}
    Define $F$ as the composition $\Rab^n \circ Q_d^m.$
    Using
    \[
        n \left( \frac{b}{a}, \frac{1}{a} \right)
        +
        m \left( 0, \frac{1}{d} \right)
        =
        \left( \frac{n b}{a}, \frac{1}{a d} \right)
        \in
        \bbR^2,
    \]
%    prlx:comment:
%        ROUGH NOTE
%
%            n d = 1 - m a
%
%        ROUGH NOTE
%
%            d left( frac{n b}{a}, frac{1}{a d} right)
%            = left( frac{n b d}{a}, frac{d}{a d} right)
%            = left( frac{(1 - m a)b}{a}, frac{1}{a} right)
%            = left( frac{b - m a b}{a}, frac{1}{a} right)
%            = left( frac{b}{a} - m b, frac{1}{a} right)
%
%
%        ROUGH NOTE
%
%            a left( frac{n b}{a}, frac{1}{a d} right)
%            = left( {n b}, frac{1}{d} right)
%            
    one can show that
    \[
        F(z,w) = (r_a^{n b}(z), r_{a d}(w)) = R_{\hat a \hat b}(z, w)
    \]
    where $\hat a = a d$ and $\hat b = n b d.$
    Further, $F^a = Q_d$ and $F^{d} = \Rab,$
    and so
    $\langle \Rab, Q_d \rangle$ is a cyclic group generated
    by $F$
    and $D^2 \ti S^1 / F$ is a standard fibered torus.
%    If H : D^2 ti S^1 -> D^2 ti S^1
%    is the $d$-fold covering map given by H(z,w) = (z, w^d).
%    One can show that H circ F = R_{a n} circ H and
%    further that H induces a diffeomorphism
%    from D^2 ti S^1 / F to D^2 ti S^1 / R_{a n}
%    which preserves the fibers.

    \smallskip{}

    Now suppose $a$ and $d$ are not coprime.
    There is $0 < n < a$ such that $n d$ is a multiple of $a$,
    and further there is $m$ such that
    \begin{math}
        n d + m a = 0.
    \end{math}
    In this case,
    \[
        \Rab^n \circ Q_d^m (z,w) = (r_a^n(z), w)
    \]
    and if $z = 0,$ then $(z,w)$ will be a fixed point.
    This shows that $\langle \Rab, Q_d \rangle$
    does not act freely on the solid torus
    and that the quotient by this group
    does not have the structure of a smooth manifold with boundary.
\end{proof}
\begin{proof}
    [Proof of \cref{prop:poscoverd}]
    Suppose $M$ is a Seifert fiber space
    as in the statement of the proposition.
    By \cref{lemma:quotienttorus} and the fact that $d$ is coprime to $\al_i$
    for all $i,$
    it follows that every standard fibered torus
    embedded in $M$ quotients down to a standard fibered torus
    in $M / (\bbZ / d).$
    Hence, the latter is also a Seifert fiber space.

    Let $R, C_i, H_i, M_i,$ and $L_i$
    be as in the construction of a Seifert fibering
    by gluing as described earlier
    and use $r, c_i, h_i, m_i, \ell_i$
    for surfaces and curves in the covered manifold $M / (\bbZ / d).$
    Let $\pi : M \to M / (\bbZ / d)$
    be the quotient map.
    One can see that
    $\pi(R) = r,$
    $\pi(C_i) = c_i,$ and
    $\pi(H_i) = d h_i$
    in homology.
    For a solid torus neighbourhood of an exceptional fiber in $M,$
    a meridional disk maps injectively down to a meridional disk
    in $M / (\bbZ / d)$
    and so $\pi(M_i) = m_i.$
    A longitude $\ell_i$ in the quotient is $d$-fold covered by
    a longitude of a solid torus in $M.$
    However, we do not know precisely which longitude
    covers $\ell_i$
    and so the most we can say
    is that $\pi(L_i + k_i M_i) = d \ell_i$ for some integer $k_i.$
    Then
    \begin{align*}
        \pi
        \begin{pmatrix}
            M_i \\ L_i \end{pmatrix}
        &=
        \begin{pmatrix}
            \al_i & \bt_i \\ \al'_i & \bt'_i \end{pmatrix}
        \ \pi
        \begin{pmatrix}
            C_i \\ H_i \end{pmatrix}
        \quad \Rightarrow \quad \\
        \begin{pmatrix}
            1 & 0 \\ -k_i & d \end{pmatrix}
        \begin{pmatrix}
            m_i \\ \ell_i \end{pmatrix}
        &=
        \begin{pmatrix}
            \al_i & \bt_i \\ \al'_i & \bt'_i \end{pmatrix}
        \begin{pmatrix}
            1 & 0 \\ 0 & d \end{pmatrix}
        \begin{pmatrix}
            c_i \\ h_i \end{pmatrix}
        \quad \Rightarrow \quad \\
        \begin{pmatrix}
            m_i \\ \ell_i \end{pmatrix}
        &=
        \begin{pmatrix}
            \al_i & d \bt_i \\ \tfrac{1}{d}(\al'_i + k_i \al_i) & \bt'_i + k_i \bt_i \end{pmatrix}
        \begin{pmatrix}
            c_i \\ h_i \end{pmatrix}
        .
    \end{align*}
    This formula would make sense whether or not
    the entries are integers
    as we are considering first homology with rational coefficients.
    However, since we already know $M / (\bbZ / d)$
    is a well-defined Seifert fiber space,
    it must be that
    $\tfrac{1}{d}(\al'_i + k_i \al_i)$ 
    is an integer.
    This shows that the covered manifold has a Seifert invariant
    as stated in the proposition.

    To prove the final statement,
    suppose $M_1 \to M_2$ is a fiberwise covering of degree $d.$
    If we equip the fibers of $M_2$ with a metric
    such that every regular fiber has length exactly $1 / d,$
    this lifts by the covering to a metric on $M_1$ where
    every regular fiber has length exactly one.
    The quotient from $M_1$ to $M_1 / (\bbZ / d)$
    is then exactly the same as the quotient from $M_1$ to $M_2.$
\end{proof}
We also wish to consider fiberwise coverings of negative degree.
First, we consider the case of $d = -1.$

\begin{prop} \label{prop:orientd}
    If $M$ has Seifert invariant
    \begin{math}
        (g; (\al_1, \bt_1), \ldots, (\al_n, \bt_n)),
    \end{math}
    then $-M$ with the same fibering has invariant
    \begin{math}
        (g; (\al_1, -\bt_1), \ldots, (\al_n, -\bt_n)).
    \end{math} \end{prop}
\begin{proof}
    Let $M \to -M$ be the orientation-reversing identity map.
    To give $-M$ the opposite orientation as $M,$
    we may assume that the orientation along the fibers has been reversed,
    but the orientation transverse to the fibers remains the same.
    Then, keeping the same notation as in the previous proof,
    $\pi(R) = r,$
    $\pi(C_i) = c_i,$
    $\pi(H_i) = -h_i,$
    $\pi(M_i) = m_i,$ and
    $\pi(L_i + k_i M_i) = -\ell_i$ for some integer $k_i.$
    One can then calculate the Seifert invariant of $-M$
    much as in the previous proof.
\end{proof}

\begin{prop} \label{prop:coverd}
    Suppose $M$ is a Seifert fiber space with invariant
    \[
        (g; (\al_1, \bt_1), \ldots, (\al_n, \bt_n))
    \]
    and $d$ is a non-zero integer coprime to $\al_i$
    for all $i.$
    Then $M$ fiberwise covers a Seifert fiber space
    with invariant
    \[
        (g; (\al_1, d \bt_1), \ldots, (\al_n, d \bt_n)).
    \]
    Moreover, any fiberwise covering of degree $d$
    between oriented
    Seifert fiber spaces is of this form.
\end{prop}
\begin{proof}
    This follows from the previous two propositions.
\end{proof}
\begin{cor} \label{cor:eulerd}
    If $M_1 \to \Sig$ and $M_2 \to \Sig$
    are Seifert fiberings and there is a fiberwise covering of degree $d$
    from $M_1$ to $M_2,$
    then
    \begin{math} 
        d \cdot e(M_1 \to \Sig) = e(M_2 \to \Sig).
    \end{math} \end{cor}

\section{General properties} \label{sec:general} %{{{1

We now consider the general properties of \hvfs{}
on Seifert fiber spaces.
Suppose $v$ is a vector field on a Seifert fiber space $M \to \Sig.$
That is, $v$ is a continuous map $v : M \to TM$ such that
$v(x)$ lies in $T_x M$ for every $x.$
We say that $v$ is \emph{horizontal} if for all $x,$ the vector
$v(x)$ is not tangent to the fiber through $x.$

\begin{prop} \label{prop:hvfmap}
    A Seifert fibering $M \to \Sig$ has a \hvf{}
    if and only if there is a continuous map $u: M \to \UTSig$
    such that
    \[
        \simpleCD{M}{u}{\UTSig}{}{}{}{\Sig}{\id}{\Sig}
    \]
    commutes.
\end{prop}    
To prove this, we first show the equivalence inside a
standard fibered torus.
Recall the definitions of a standard fibered torus
\begin{math}
    \Uab = D^2 \ti S^1 / \Rab
\end{math}
and a quotiented disk
\begin{math}
    D_a = D^2 / r_a.
\end{math}
Both $\Uab$ and the unit tangent bundles $\UT D_a$
have fiberings given by canonical maps $\pi : \Uab \to D_a$
and $\pi_a : \UT D_a \to D_a.$

\begin{lemma} \label{lemma:localunittan}
    Any horizontal vector field on a standard fibered torus
    $\Uab$ induces a continuous fiber-preserving map from $\Uab$ to
    the unit tangent bundle of \, $D_a$.
\end{lemma}
\begin{proof}
    A vector field $v$ on $\Uab$ may be considered as
    a continuous function
    $v : \Uab \to T \Uab.$
    Lifting the vector field to $D^2 \ti S^1,$
    we have a function
    $v_1 : D^2 \ti S^1 \to T(D^2 \ti S^1)$
    which satisfies
    \[
        v_1 \circ \Rab = T \Rab \circ v_1.
    \]
    The lifted vector field is horizontal
    and for a point $y = (z, w) \in D^2 \ti S^1$
    the vector $v_1(y)$ is not tangent to the fiber
    $\{z\} \ti S^1.$
    In particular, if $\Pi : D^2 \ti S^1 \to D^2$
    is the projection onto $D^2,$
    then 
    $T \Pi (v_1(y)) \in T_z D^2$
    is non-zero.
    Define a map
    \[    
        u_2 : D^2 \ti S^1 \to \UT D^2,
        \quad
        y \mapsto \normalize{ T \Pi(v_1(y)) }.
    \]
    One may show, using the definition of $\Rab,$
    that $u_2$ satisfies the symmetry relation
    \begin{math}
        T r^b_a \circ u_2 = u_2 \circ T \Rab.
    \end{math}
    Hence,
    $u_2$ quotients to a map $u$ from
    $\Uab = D^2 \ti S^1 / \Rab$
    to
    $\UT D_a = \UT D^2 / T r_a.$
    Moreover, as $u_2$ takes fibers to fibers,
    so does $u.$
\end{proof}
Note that the above proof works for any choice of Riemannian metric
on $D_a.$

\begin{lemma} \label{lemma:lifttohvf}
    Let $\Uab$ be a standard fibered torus,
    $P$ a plane field on $\Uab$ transverse to the fibering,
    and
    $u : \Uab \to \UT D_a$
    a continuous function for which 
    \[    
        \simpleCD{\Uab}{u}{\UT D_a}{\pi}{}{\pi_a}{\Sig}{\id}{\Sig}
    \]
    commutes.
    Then there exists a unique \hvf{} $v : \Uab \to T \Uab$
    such that for all $x \in \Uab,$
    $v(x) \in P(x)$
    and
    $T \pi(v(x)) = u(x).$
\end{lemma}
\begin{proof}
    Each of $\Uab$ and $\UT D_a$ is a solid torus
    with an infinite cyclic fundamental group.
    Since both covering maps $D^2 \ti S^1 \to \Uab$
    and $\UT D^2 \to \UT D_a$
    are of the same degree $a$,
    one may verify that the lifting criterion
    \cite[Proposition 1.33]{hat2002algebraic}
    applies here and $u : \Uab \to \UT D_a$
    lifts to a map $u_1 : D^2 \ti S^1 \to \UT D^2.$
    Moreover, we may choose the lift in such a way that
    \[
        \simpleCD{D^2 \ti S^1}{u_1}{\UT D^2}{}{}{}{D^2}{\id}{D^2}
    \]
    commutes.
    To see this, first choose a lift so that the above
    diagram commutes for at least one point $y \in D^2 \ti S^1,$
    and then use that $D^2 \ti S^1$ is connected
    to show that it commutes for all points.
    The lifted map satisfies
    \begin{math}
        u_1 \circ \Rab = T r^n_a \circ u_1
    \end{math}
    for some integer $n,$ and the above commutative diagram
    shows that $n \equiv b \bmod a.$
    Without loss of generality, assume $n = b.$

    The transverse plane field $P$ on $\Uab$
    lifts to a transverse plane field $P_1$ on $D^2 \ti S^1$
    which is invariant under $T \Rab.$
    Let $\Pi : D^2 \ti S^1 \to D^2$ be the canonical projection.
    By transversality, for each point $y = (z, w) \in D^2 \ti S^1$
    the tangent map
    \[
        T_y \Pi : T_y (D^2 \ti S^1) \to T_z D^2
    \]
    restricts to an isomorphism from the 2-dimensional subspace
    $P(y)$ onto $T_z D^2.$
    Define $u_2(y)$ as the unique vector in $P(y)$
    which satisfies $T \Pi (v_2(y)) = u_1(y) \in T_z D^2.$
    Doing this at all points in $D^2 \ti S^1$ defines a vector field
    $v_2 : D^2 \ti S^1 \to T (D^2 \ti S^1).$

    Using $P(\Rab(y)) = T \Rab(P(y))$
    and
    $u_1(\Rab(y)) = T r^b_a u_1(y),$
    one can see that
    $v_2(\Rab(y)) = T \Rab v_2(y)$).
    Hence,
    $u_2$ quotients down to a vector field $u : \Uab \to T \Uab$
    and this is the desired vector field.
    Since $u_2$ is uniquely determined, so is $u.$
\end{proof}
\begin{proof}
    [Proof of \cref{prop:hvfmap}]
    Suppose $v : M \to TM$ is a \hvf{}.
    Define a function
    \[    
        u : M \to \UTSig, \quad x \mapsto \normalize{ T \pi(v(x)) }
    \]
    where $\pi : M \to \Sig$ defines the fibering.
    As $v$ is horizontal, $u$ is well defined.
    Any point $x \in M$ lies inside of a standard fibered torus,
    where the definition of $u$ agrees with the map given in
    \cref{lemma:localunittan}.
    Hence, $u$ is continuous and is the desired map.

    For the converse direction, suppose $u : M \to \UTSig$
    is a map as in the statement of the proposition.
    Choose a plane field $P$ transverse to the fibering.
    Such a plane field always exists;
    for instance, one can equip $M$ with a Riemannian metric
    and consider the planes which are orthogonal to the fibers.
    Then define $v : M \to TM$ as the unique vector field
    tangent to $P$
    which satisfies $T \pi(v(x)) = u(x)$ for all $x \in M.$
    By \cref{lemma:lifttohvf},
    such a $v$ exists and is continuous.
    Moreover, $v$ is horizontal by construction.
\end{proof}

Fiber-preserving maps must be of a specific form.

\begin{prop} \label{prop:seicover}
    Suppose $M_1$ and $M_2$ are Seifert fiber spaces over the same
    base orbifold $\Sig$
    and
    $u:M_1 \to M_2$ is a fiber-preserving map such that
    \[
        \simpleCD{M_1}{u}{M_2}{}{}{}{\Sig}{\id}{\Sig}
    \]
    commutes.
    Then either
    \begin{enumerate}
        \item        %$deg u = 0$ and
        $u$ is homotopic to a composition of the form
        $M_1 \to \Sig \to M_2$, or
        \item
        %$deg u != 0$ and
        $u$ is homotopic to a fiberwise covering $M_1 \to M_2$.
    \end{enumerate} \end{prop}
\begin{remark}
    In the first case, the map $M_1 \to \Sig$ is the projection
    defining the fibering and the second map $\Sig \to M_2$
    is a continuous section for the fibering $M_2 \to \Sig$.
    In either case, the homotopy is along the fibers of $M_2$,
    as may be seen from the proof below.
\end{remark}
\begin{proof}
    First assume that $M_1$ and $M_2$ and their fibers are orientable.
    Because $u$ projects down to the identity map
    on $\Sig,$ the degree of $u$ is equal to the degree of the restriction
    of $u$ to any fiber of $M_1.$
    Further, since the fiberings of $M_1$ and $M_2$ are smooth,
    we may perturb the continuous map $u$ slightly to produce a smooth function
    which also projects to the identity on $\Sig.$
    This smooth perturbation is homotopic to $u,$
    and so without loss of generality, we assume $u$ itself is smooth.

    Equip each of $M_1$ and $M_2$ with a metric so that every regular
    fiber has length exactly one.
    The exceptional fibers will then have lengths $1/\al_i$
    for integers $\al_i>1$.
    Consider a point $x \in M_1$ and
    let $L_1$ be the fiber through $x$
    and $L_2$ the fiber through $u(x)$.
    Define length-preserving covering maps
    $\pi_1: \bbR \to L_1$ and
    $\pi_2: \bbR \to L_2$
    and choose a lift $\tilde u: \bbR \to \bbR$
    such that $\pi_2 \circ \tilde u = u \circ \pi_1 $.
    Define $\hat u: \bbR \to \bbR$ by
    \[
        \hat u(t) = \int_{t-\tfrac{1}{2}}^{t+\tfrac{1}{2}} \tilde u(s) \, ds
        \qquad \text{so that} \qquad
        \frac{d \hat u}{d t}
        = \tilde u \left(t+\tfrac{1}{2}\right)
        - \tilde u \left(t-\tfrac{1}{2}\right) = \deg(u).
    \]
    There is then a unique function $h:L_1 \to L_2$ such that
    $h \circ \pi_1 = \pi_2 \circ \hat u$
    and $h$ is independent of the choice of lift $\tilde u$.
    When defined fiber by fiber on all of $M_1$,
    this gives a smooth map $h:M_1 \to M_2$.
    By shrinking the interval $[t-1/2,t+1/2]$ down to a point,
    one can show that $h$ is homotopic to $u$
    and that the homotopy is through fiber-preserving maps.

    If $\deg u = 0$, then $h$ is constant on every fiber of $M_1,$
    and so it factors as a composition $M_1 \to \Sig \to M_2.$
    If $\deg u \ne 0$, then $h$ is a covering map on each fiber.
    Since $h$ projects to the identity on $\Sig,$
    one can see that it is a global covering map $M_1 \to M_2.$

    If any of $M_1, M_2,$ or the fiberings is not orientable,
    the steps above still work.
    However, the degree of $u$ is only defined up to a sign.
    The interval $[t-1/2,t+1/2]$ was used above in order to
    avoid having to choose an orientation on a fiber.
\end{proof}
Because of \cref{prop:hvfmap},
we primarily consider \cref{prop:seicover} in the case where
$M_2$ is a unit tangent bundle.
When is the case $M \to \Sig \to \UTSig$ possible?
This requires a horizontal section $\Sig \to \UTSig$;
that is, a vector field on $\Sig$ which by \cref{prop:orbvfield}
is possible only if $\Sig$ is the 2-torus or Klein bottle.
This situation is discussed in further detail in \cref{sec:para}.
Outside of this case, we have the following equivalences.

\begin{thm} \label{thm:bigequiv}
    Let $M \to \Sig$ be a Seifert fibering over
    a base orbifold $\Sig$ which is neither the 2-torus
    nor the Klein bottle.
    Then the following are equivalent:
    \begin{enumerate}
        \item $M$ has a \hvf{};
        \item
        there is a continuous map $u: M \to \UTSig$
        such that
        \[
            \simpleCD{M}{u}{\UTSig}{}{}{}{\Sig}{\id}{\Sig}
        \]
        commutes;
        \item
        there is a fiberwise covering from $M$ to
        the unit tangent bundle $\UTSig;$
        \item
        $M$ is orientable, the fibering has Seifert invariant
        \[
            (g; (\al_1, \bt_1), \ldots, (\al_k, \bt_k))
        \]
        and there is a non-zero integer $d$
        such that
        \begin{math}
            d \cdot e(M \to \Sig) = \chi(\Sig)
        \end{math}
        and
        \[        
            d \bt_i/\al_i \equiv -1/\al_i    \mod \bbZ
        \]
        for all $i = 1, \ldots, k.$
    \end{enumerate} \end{thm}
\begin{proof}
    The equivalence
    $(1)$ $\Leftrightarrow$ $(2)$
    is given by \cref{prop:hvfmap}
    and
    $(3)$ $\Leftrightarrow$ $(4)$
    is a combination of
    propositions \ref{prop:UTinvt} and \ref{prop:coverd}.
    Note that as $\UTSig$ is oriented,
    any manifold $M$ covering $\UTSig$ is orientable.
    The direction $(2)$ $\Leftarrow$ $(3)$
    follows immediately from the definition of a fiberwise covering.
    The only place we use that $\Sig$ is not the 2-torus or Klein bottle
    is to invoke \cref{prop:orbvfield} to ensure we are in the second case
    of \cref{prop:seicover}.
    This gives $(2)$ $\Rightarrow$ $(3).$
\end{proof}
The integer $d$ in item $(4)$ above has topological significance
in the other three items as well.
In $(2)$ and $(3), d$ is the degree of the map $u : M \to \UTSig.$
In $(1), d$ is the number of turns that the vector field makes
while going around a regular fiber.

In the next four sections,
we consider the consequences of \cref{thm:bigequiv}
for the cases of bad, elliptic, parabolic, and hyperbolic
base orbifolds respectively.

\section{Lens spaces} \label{sec:lens} %{{{1

We now consider Seifert fiberings on lens spaces.
To aid in the study, we introduce in this paper
the notion of a ``marked'' lens space.

For a solid torus, $D^2 \ti S^1,$
the \emph{core} of a solid torus
is the curve $\{0\} \ti S^1$
where $0$ is the center of the disk $D^2.$
A \emph{lens space}
is a closed oriented 3-manifold $L$ constructed
by gluing together together two solid tori
$U_1$ and $U_2$
by a linear map $A : \bbT^2 \to \bbT^2$
which identifies the two boundaries.
A \emph{marked lens space}
is a lens space along with a choice of orientation
for each of the cores of the two solid tori.
A \emph{fibered marked lens space}
is a marked lens space equipped with a Seifert fibering
such that each of $U_1$ and $U_2$ are standard fibered tori.

\medskip{}

Later in this section,
we show that there are two families of Seifert fiberings on lens spaces.
One of these is the fibered marked lens spaces.
The other is produced by taking a single standard fibered torus and,
by identifying points on the boundary,
closing it up into a closed 3-manifold.
We first fully treat the fibered marked lens spaces before handling
this second case.

\subsection{Fibered marked lens spaces}

Now consider a marked lens space $L.$
For the boundaries of the two solid tori,
the meridians $M_i$ are well defined
(at least up to homology)
as they bound solid disks.
The longitudes, however,
are not well defined.
If $L_i$ is a longitude on the boundary of $U_i,$
then for any integer $k_i,$
there is a diffeomorphism from $U_i$ to itself
which maps $L_i$ to a curve homologous
to $L_i + k_i M_i.$
As we are assuming the core of $U_i$ has a
specified orientation,
we only consider longitudes with a compatible orientation.

Viewing the gluing map $A$ as two-by-two matrix,
we have
\[
    \begin{pmatrix}
        M_1 \\
        L_1 \end{pmatrix}
    =
    \ A \
    \begin{pmatrix}
        M_2 \\
        L_2
    \end{pmatrix}  \]
in homology.
In order for the lens space $L$ to be an oriented 3-manifold
the gluing map must reverse the orientation of the boundary,
and so this matrix has a determinant equal to $-1$.
In this section, we write this gluing matrix as
\[
    A
    =
    \begin{pmatrix}
        -q_2 & p \\
        * & q_1
    \end{pmatrix}  \]
with a $*$ for the lower-left entry.
We omit this entry as it is uniquely determined
by $\det(A) = -1$ and the other three entries,
and whereas the other entries have topological
significance, this lower-left entry
is usually given by a relatively complicated formula
with no particular significance.

The matrix depends on the choice of longitudes $L_i + k_i M_i$
and a different choice of longitudes would yield a matrix
\[
    \begin{pmatrix}
        1 & 0 \\
        k_1 & 1 \end{pmatrix}
    \begin{pmatrix}
        -q_2 & p \\
        * & q_1 \end{pmatrix}
    \begin{pmatrix}
        1 & 0 \\
        k_2 & 1 \end{pmatrix}
    \inv
    =
    \begin{pmatrix}
        -q_2-p k_2 & p \\
        * & q_1 + p k_1 \end{pmatrix}
    .
\]
This shows that the integer $p$
is well defined
and that $q_1$ and $q_2$ are well defined modulo $p.$
If we relabelled $U_1$ and $U_2$
as $U_2$ and $U_1$ respectively,
then the gluing map would be replaced by its inverse,
and the matrix would then be
\[
    \begin{pmatrix}
        -q_2 & p \\
        * & q_1 \end{pmatrix}
    \inv
    =
    \begin{pmatrix}
        -q_1 & p \\
        * & q_2 \end{pmatrix}
    .
\]
Note also that the condition $\det(A) = -1$
implies that $q_1 q_2 \equiv -1 \bmod p$.

For coprime integers $p$ and $q,$ define $L(p,q)$
as the marked lens space
constructed by gluing together two solid tori
by a gluing map with a matrix of the form
\[
    \begin{pmatrix}
        -q_2 & p \\
        * & q_1
    \end{pmatrix}  \]
where $q_1 = q$ and the first column is chosen so the determinant is $-1$.
As the choice of the first column
of the matrix corresponds only to a choice
of longitude on $U_1,$ the marked lens space $L(p,q)$
is well defined.
Further, for any integer $k,$
$L(p, q)$ and  $L(p, q + p k)$ are the same marked lens space.
If 
$q_1 q_2 \equiv -1 \bmod p$.
then
$L(p, q_1)$ and $L(p, q_2)$
are the same marked lens spaces,
differing only by the choice of labelling
of the two solid tori as $U_1$ and $U_2.$

Now suppose we take a marked lens space
and produce a new marked lens space by
reversing the orientation of the core of exactly
one of the two solid tori,
say $U_2.$
This reverses the orientation of the longitude $L_2$ and,
in order to keep the original orientation of the 3-manifold
unchanged,
we must reverse the orientation of the meridian $M_2$ as well.
With respect to these new orientations,
the gluing matrix is now given by
\[
    \begin{pmatrix}
        -q_2 & p \\
        * & q_1 \end{pmatrix}
    \begin{pmatrix}
        -1 & 0 \\
        0 & -1 \end{pmatrix}
    =
    \begin{pmatrix}
        q_2 & -p \\
        * & -q_1 \end{pmatrix}
    .
\]
This leads to the following key observation:
\begin{quote}
    $L(p,q)$ and $L(-p,-q)$
    are diffeomorphic as oriented 3-manifolds,
    but they are not the same \emph{marked} lens space.
\end{quote}

We can reverse the orientation of a lens space,
say by leaving the orientations of both cores unchanged,
but reversing the orientations of both meridians.
The new gluing matrix is
\[
    \begin{pmatrix}
        1 & 0 \\
        0 & -1 \end{pmatrix}
    \begin{pmatrix}
        -q_2 & p \\
        * & q_1 \end{pmatrix}
    \begin{pmatrix}
        1 & 0 \\
        0 & -1 \end{pmatrix}
    =
    \begin{pmatrix}
        -q_2 & -p \\
        * & q_1
    \end{pmatrix}  \]
from which one sees that $-L(p,q) = L(-p,q).$

\medskip{}

Lens spaces give examples of manifolds which are homotopy equivalent,
but not homeomorphic \cite{bro1960topological}.
In fact, it is a highly non-trivial result
that the only homeomorphisms between lens spaces
are the ``obvious'' ones explained above.

\begin{thm}
    [Brody \cite{bro1960topological}] \label{thm:lenshomeo}
    Lens spaces $L(p_1, q_1)$ and $L(p_2, q_2)$
    are homeomorphic
    if and only if $|p_1| = |p_2|$
    and $q_1 \equiv \pm q_2^{\pm 1} \bmod p.$
\end{thm}
Now consider a Seifert fiber space of the
form $M(0; (\al_1, \bt_1), (\al_2, \bt_2)).$
This is constructed by gluing two standard fibered tori
into a manifold with boundary $M_0 = F_0 \ti S^1$
where $F_0$ is a sphere with two disks removed and so is an annulus.
As $M_0$ is an $I$-bundle,
this is effectively the same as gluing the boundaries of the
two solid tori together by a linear map that respects the fiberings,
and we may think of the Seifert fiber space as a fibered marked lens space.

\begin{thm} \label{thm:lenspq}
    The Seifert invariant $(0; (\al_1, \bt_1), (\al_2, \bt_2))$ yields 
    a fibered marked lens space with gluing matrix
    \[
        \begin{pmatrix}
            -\al_1 \bt'_2 - \al'_2 \bt_1 & \al_1 \bt_2 + \al_2 \bt_1 \\
            * & \al'_1 \bt_2 + \al_2 \bt'_1 \end{pmatrix}
        \quad \text{where} \quad
        \begin{pmatrix}
            \al_i & \bt_i \\
            \al'_i & \bt'_i \end{pmatrix}
        \in
        SL(2, \bbZ).
    \]

    In particular, the marked lens space is $L(p,q)$
    with
    \[
        p = \al_1 \bt_2 + \al_2 \bt_1 \qandq
        q = \al'_1 \bt_2 + \al_2 \bt'_1.
    \] \end{thm}
%    
%        M_i ~ al_i C_i + bt_i H_i
%        qandq
%        L_i ~ al'_i C_i + bt'_i H_i

\begin{proof}
    Recall in the construction of a Seifert fiber space by gluing
    that
    \[
        \begin{pmatrix}
            M_i \\
            L_i \end{pmatrix}
        =
        \begin{pmatrix}
            \al_i & \bt_i \\
            \al'_i & \bt'_i \end{pmatrix}
        \begin{pmatrix}
            C_i \\
            H_i
        \end{pmatrix}  \]
    where $H_i$ is a vertical fiber,
    $C_i$ is a horizontal circle,
    and $\al'_i$ and $\bt'_i$ are chosen to yield so that
    the matrix lies in $SL(2, \bbZ).$

    Note that we are not free to replace $L_i$ with $-L_i$ here
    without changing the gluing.
    Hence, the pair $(\al_i, \bt_i)$ in the Seifert invariant
    determines an orientation of the longitude $L_i$
    and so also determines an orientation of the core of the solid torus.
    This means we are in fact constructing a marked lens space.

    As the two circles $C_1$ and $C_2$ are the boundaries of two disks
    punched out of a sphere, $C_1 \sim -C_2$ in homology.
    As regular vertical fibers, $H_1 \sim H_2$ in homology.
    Hence,
    \begin{align*}
        \begin{pmatrix}
            C_1 \\ H_1 \end{pmatrix}
        =&
        \begin{pmatrix}
            -1 & 0 \\
            0 & 1 \end{pmatrix}
        \begin{pmatrix}
            C_2 \\ H_2 \end{pmatrix}
        \quad \Rightarrow \quad \\
        \begin{pmatrix}
            \al_1 & \bt_1 \\
            \al'_1 & \bt'_1 \end{pmatrix}
        \inv    
        \begin{pmatrix}
            M_1 \\ L_1 \end{pmatrix}
        =&
        \begin{pmatrix}
            -1 & 0 \\
            0 & 1 \end{pmatrix}
        \begin{pmatrix}
            \al_2 & \bt_2 \\
            \al'_2 & \bt'_2 \end{pmatrix}
        \inv
        \begin{pmatrix}
            M_2 \\ L_2 \end{pmatrix}
        \quad \Rightarrow \quad \\
        \begin{pmatrix}
            M_1 \\ L_1 \end{pmatrix}
        =&
        \begin{pmatrix}
            \al_1 & \bt_1 \\
            \al'_1 & \bt'_1 \end{pmatrix}
        \begin{pmatrix}
            -1 & 0 \\
            0 & 1 \end{pmatrix}
        \begin{pmatrix}
            -\bt'_2 & \bt_2 \\
            -\al'_2 & \al_2 \end{pmatrix}
        \begin{pmatrix}
            M_2 \\ L_2 \end{pmatrix}
        \quad \Rightarrow \quad \\
        \begin{pmatrix}
            M_1 \\ L_1 \end{pmatrix}
        =&
        \begin{pmatrix}
            -\al_1 \bt'_2 - \al'_2 \bt_1 & \al_1 \bt_2 + \al_2 \bt_1 \\
            * & \al'_1 \bt_2 + \al_2 \bt'_1 \end{pmatrix}
        \begin{pmatrix}
            M_2 \\ L_2 \end{pmatrix}
        .
        \qedhere
    \end{align*} \end{proof}
%    This shows that the marked lens space is L(p,q)
%    with
%    p = al_1 bt_2 + al_2 bt_1 and
%    q = al'_1 bt_2 + al_2 bt'_1.

Using this, we may calculate which lens spaces arise as unit tangent bundles.

\begin{cor} \label{cor:unitlens}
    If $\Sig$ is a sphere with $0, 1,$ or $2$ cone points added,
    then $\UTSig$ is a marked lens space of the form $L(p, -1)$
    with $p \le -2.$
\end{cor}
\begin{proof}
    By \cref{prop:UTinvt}, we may write the fibering as
    \begin{math}
        %UTSig =
        M(0; (\al_1, -1), (\al_2, -1))
    \end{math}
    where $\al_1 = \al_2 = 1$ if $\Sig = S^2$ is the two sphere
    and $\al_1 = 1$ if $\Sig$ has a single cone point.
    One can then apply \cref{thm:lenspq} with $\al'_i = 1$ and $\bt'_i = 0$
    to get the result.
    Alternatively, it is fairly easy
    to calculate the gluing matrix directly.
    Adapting the proof of \cref{thm:lenspq}, the matrix is given by
    \[
        \begin{pmatrix}
            \al_2 & -1 \\
            1 & 0 \end{pmatrix}
        \begin{pmatrix}
            -1 & 0 \\
            0 & 1 \end{pmatrix}
        \begin{pmatrix}
            \al_1 & -1 \\
            1 & 0 \end{pmatrix}
        \inv
        =
        \begin{pmatrix}
            1 & -\al_1-\al_2 \\
            * & -1 \end{pmatrix}
        .
    \]
    Since $\al_1, \al_2 \ge 1,$ it follows that $p = -\al_1-\al_2 \le -2.$
\end{proof}
\begin{prop} \label{prop:lenscover}
    Let $\hat L$ and $L$ be fibered marked lens spaces
    and let $\pi : \hat L \to L$ be a fiberwise
    covering of degree $d.$
    Then there are integers $p$ and $q$ such that
    $\hat L = L(p, q)$
    and
    $L = L(d p, q).$
\end{prop}
\begin{proof}
    We use the same notation as in the proofs of
    propositions \ref{prop:poscoverd} and \ref{prop:orientd}.
    It follows from those proofs that
    $\pi(M_i) = m_i$
    and $\pi(L_i + k_i M_i) = d \ell_i$ for some integer $k_i.$
    Then,
    \begin{align*}
        \pi
        \begin{pmatrix}
            M_1 \\ L_1 \end{pmatrix}
        =&
        \begin{pmatrix}
            -q_2 & p \\
            * & q_1 \end{pmatrix}
        \ \pi
        \begin{pmatrix}
            M_2 \\ L_2 \end{pmatrix}
        \quad \Rightarrow \quad \\
        \begin{pmatrix}
            1 & 0 \\
            -k_1 & d \end{pmatrix}
        \begin{pmatrix}
            m_1 \\ \ell_1 \end{pmatrix}
        =&
        \begin{pmatrix}
            -q_2 & p \\
            * & q_1 \end{pmatrix}
        \begin{pmatrix}
            1 & 0 \\
            -k_2 & d \end{pmatrix}
        \begin{pmatrix}
            m_2 \\ \ell_2 \end{pmatrix}
        \quad \Rightarrow \quad \\
        \begin{pmatrix}
            m_1 \\ \ell_1 \end{pmatrix}
        =&
        \frac{1}{d}
        \begin{pmatrix}
            d & 0 \\
            k_1 & 1 \end{pmatrix}
        \begin{pmatrix}
            -q_2 & p \\
            * & q_1 \end{pmatrix}
        \begin{pmatrix}
            1 & 0 \\
            -k_2 & d \end{pmatrix}
        \begin{pmatrix}
            m_2 \\ \ell_2 \end{pmatrix}
        \quad \Rightarrow \quad \\
        \begin{pmatrix}
            m_1 \\ \ell_1 \end{pmatrix}
        =&
        \begin{pmatrix}
            -q_2 - k_2 p & d p \\
            * & q_1 + k_1 p \end{pmatrix}
        \begin{pmatrix}
            m_2 \\ \ell_2
        \end{pmatrix} \end{align*}
    Hence, the claim holds with $q = q_1 + k_1 p.$
\end{proof}
Note that $q \bmod d p$ uniquely determines $q \bmod p,$
and so the covered lens space uniquely determines the
covering lens space.

\begin{prop} \label{prop:allfiberings}
    A fibered marked lens space $L(p,q)$ 
    has a \hvf{} if and only if
    $p \ne 0$ and $q \equiv -1 \bmod p.$
\end{prop}
\begin{proof}
    First assume $L(p,q)$ has a \hvf{}.
    Then by \cref{thm:bigequiv} and \cref{cor:unitlens},
    it fiberwise covers $L(d p, -1)$ for some integer $d$
    with the further property that $d p \le -2.$
    This implies that $p$ must be non-zero.
    By \cref{prop:lenscover}, $L(p,q) = L(p, -1)$ as a marked lens space
    and so $q \equiv -1 \bmod p.$

    Conversely, assume $p \ne 0$ and $q \equiv -1 \bmod p.$
    As explained earlier in this section,
    we may make choices of longitudes,
    or equivalently choices of $\al'_i$ and $\bt'_i,$
    such that the gluing matrix is
    \[
        \begin{pmatrix}
            \al_1 & \bt_1 \\
            \al'_1 & \bt'_1 \end{pmatrix}
        \begin{pmatrix}
            -1 & 0 \\
            0 & 1 \end{pmatrix}
        \begin{pmatrix}
            \al_2 & \bt_2 \\
            \al'_2 & \bt'_2 \end{pmatrix}
        \inv
        =
        \begin{pmatrix}
            1 & p \\
            * & -1 \end{pmatrix}
        .
    \]
    Then $\al_1 \bt'_1 - \al'_1 \bt_1 = 1$
    as it is the determinant of a matrix in $SL(2, \bbR).$
    \Cref{thm:lenspq} implies that
    $\al'_1 \bt_2 + \al_2 \bt'_1 = -1.$
    These equations may be rewritten as
    \[
        \al'_1 \left( \frac{-\bt_1}{\al_1} \right) = \frac{1}{\al_1} - \bt'_1
        \qandq
        \al'_1 \left( \frac{-\bt_2}{\al_2} \right) = \frac{1}{\al_2} + \bt'_1.
    \]
    One may then show that \cref{prop:coverd} holds with $d = \al'_1.$
    Indeed,
    $d(-\bt_i/\al_i) \equiv 1/\al_i \bmod \bbZ$
    since $\bt'_1$ is an integer, and summing the above equations gives
    \[
        d \cdot e(M \to \Sig)
        =
        d \left( - \frac{\bt_1}{\al_1} - \frac{\bt_2}{\al_2} \right)
        =
        \frac{1}{\al_1} + \frac{1}{\al_2}
        =
        \chi(\Sig).
        \qedhere
    \] \end{proof}
Before proceeding, we attempt here to give a more intuitive
and topological explanation of why $q \equiv -1 \bmod p$
is a necessary and sufficient condition for a \hvf{}
and why $d = \al'_1$ in the proof above.
Since the results have already been rigorously proved above,
we do not give complete details in the following explanation.

First, suppose that we have Seifert fibering on a marked lens space such that
the gluing torus is a union of fibers and that there is a \hvf{}.
For any closed curve on the gluing torus,
we can measure how many turns the \hvf{} makes as we travel around this curve.
This may be done, say, by counting how many times the vector field
crosses the gluing torus and changes from pointing into one solid torus
to pointing into the other.
This counting of turns induces a group homomorphism
$\phi : H_1(T) \to \bbZ$
where $H_1(T)$ is the first homology group of the gluing torus.
Since $M_i$ bounds a meridional disk and the \hvf{} projects
to a non-zero vector field on the disk, the Poincar\'e-Hopf theorem
implies that the vector field must make exactly one clockwise turn
when going around $M_i.$ Hence $\phi(M_i) = -1.$
In fact, this is only restriction on $\phi.$ One can check that
for any \hvf{} defined on the boundary of a standard fibered torus,
that so long as $\phi(M_i) = -1$ it is possible to extend the \hvf{} to all of the
solid torus.

The relation
\[
    \begin{pmatrix}
        M_1 \\
        L_1 \end{pmatrix}
    =
    \begin{pmatrix}
        -q_2 & p \\
        * & q_1 \end{pmatrix}
    \begin{pmatrix}
        M_2 \\
        L_2
    \end{pmatrix}  \]
implies that $M_1 \sim -q_2 M_2 + p L_1$
and so
\[
    -1 = \phi(M_1) = -q_2 \, \phi(M_2) + p \, \phi(L_1)
    = q_2 + p \, \phi(L_1).
\]
That is, $q_2 \equiv -1 \bmod p$ and the marked lens space is $L(p, -1).$

Conversely, suppose we have a Seifert fibering on $L(p, -1)$
and we wish to construct a \hvf{}.
It is enough to find a homomorphism $\phi : H_1(T) \to \bbZ$
such that $\phi(M_1) = \phi(M_2) = -1.$
Up to choosing different longitudes on the solid tori,
we may assume that the gluing matrix is of the form
\[
    \begin{pmatrix}
        M_1 \\
        L_1 \end{pmatrix}
    =
    \begin{pmatrix}
        1 & p \\
        0 & -1 \end{pmatrix}
    \begin{pmatrix}
        M_2 \\
        L_2 \end{pmatrix}
    .
\]
Under such a choice, we may then define $\phi$ by $\phi(M_i) = -1$
and $\phi(L_i) = 0$ for $i = 1,2$ and see that this agrees with the gluing.
In the construction of a Seifert fibering by gluing,
the cycles in homology satisfy
\[
    \begin{pmatrix}
        C_i \\
        H_i \end{pmatrix}
    \, = \,
    \begin{pmatrix}
        \al_i & \bt_i \\
        \al'_i & \bt'_i \end{pmatrix}
    \inv
    \!\!
    \begin{pmatrix}
        M_i \\
        L_i \end{pmatrix}
    \, = \,
    \begin{pmatrix}
        \bt'_i & -\bt_i \\
        -\al'_i & \al_i \end{pmatrix}
    \begin{pmatrix}
        M_i \\
        L_i
    \end{pmatrix}  \]
and so $\phi(H_i) = -\al'_i \phi(M_i) + \al_i \phi(L_i) = \al'_i.$
Since $H_i$ is a regular fiber,
$\phi(H_i)$ measures the number of twists the vector field
makes around the regular fiber and so equals the degree $d$
of the covering $M \to \UTSig.$
This shows that $d = \al'_1 = \al'_2.$
This ends our discussion of the proof of \cref{prop:allfiberings}.
    
\medskip{}

We now prove a version of \cref{thm:lens} for fibered marked lens spaces.

\begin{prop} \label{prop:marklens}
    Suppose $M$ is diffeomorphic to the lens space $L(p,q)$ with $p \ge 0$
    and consider only the Seifert fiberings on $M$ which
    have an invariant of the form $(0; (\al_1, \bt_1), (\al_2, \bt_2)).$
    \begin{enumerate}
        %If p = 0, then no Seifert fibering on M
        %has a hvf().
        %---
        \item        %If 1 <= p <= 2,
        If $p = 1$ or $p = 2$,
        then every such fibering on
        $M$ has a \hvf{}.
        \item
        If $p \ge 3$ and $q \equiv \pm 1 \bmod p$,
        then $M$ has infinitely many fiberings of this form
        which support \hvfs{} and infinitely many which do not.
        \item
        For all other cases of $p$ and $q,$
        no fibering of this form on $M$ has a \hvf{}.
    \end{enumerate} \end{prop}    
\begin{proof}
    We prove this using \cref{thm:lenshomeo} and \cref{prop:allfiberings}.
    If $p = 0,$ then the only marked lens space diffeomorphic to $M$ is
    $L(0,1) = S^2 \ti S^1$ and \cref{prop:allfiberings} implies that no
    fibering has a \hvf{}.

    If $p = 1$ or $p = 2,$ 
    the fact that $q$ is coprime to $p$ implies that $q \equiv -1 \bmod p.$
    That is, the only marked lens space diffeomorphic to $M$
    is $L(p, -1)$ and so every fibering has a \hvf{}.

    If $M$ is diffeomorphic to $L(p,q)$ with
    $|p| > 3$ and $q \equiv -1 \bmod p,$
    then $M$ is diffeomorphic to both $L(p,-1)$ and $L(p,+1).$
    The former has infinitely many fiberings,
    all of which have a \hvf{};
    the latter has infinitely many fiberings,
    none of which has a \hvf{}.

    Finally,
    if $q \not\equiv \pm 1 \bmod p,$
    then \cref{thm:lenshomeo} shows that $M$ is not diffeomorphic to $L(p, -1)$
    and so no fibering of the form $(0; (\al_1, \bt_1), (\al_2, \bt_2))$
    will have a \hvf{}.
\end{proof}    

\subsection{Other fiberings of lens spaces}

To complete the proof of \cref{thm:lens},
we must consider all possible Seifert fiberings of lens spaces.
To do this, we first list out all closed 3-manifolds
which have multiple possible Seifert fiberings.

\begin{prop} \label{prop:multifiber}
    A Seifert fibering of an oriented 3-manifold (without boundary)
    is unique up to isomorphism
    with the following exceptions:
    \begin{enumerate}
        \item Every lens space has infinitely many different Seifert fiberings,
        and
        any Seifert fibering of a lens space
        is either of the form
        \[
            (0; (\al_1, \bt_1), (\al_2, \bt_2))
            \qorq
            (-1; (\al, \pm 1)).
        \]
        \item
        For a non-zero rational number p/q written in lowest terms,
        write $\al_1 = |p|,$
        $\al_3 = |q|,$
        and define $\bt_1$ and $\bt_3$ such that
        $\bt_3/\al_3 = \al_1/\bt_1 = p/q,$
        then the invariants
        \[
            (0; (2, 1), (2, -1), (\al_3, \bt_3))
            \qandq
            (-1; (\al_1, \bt_1))
        \]
        give fiberings of the same manifold.
        This manifold is a lens space if and only if $\al_3 = 1.$
        \item
        The unit tangent bundles of the $2222$ orbifold
        and the Klein bottle
        are diffeomorphic as manifolds.
    \end{enumerate} \end{prop}
\begin{proof}
    This is a re-statement of a known result on the uniqueness of Seifert fiberings.
    As a reference, we give the lecture notes of Hatcher
    \cite[Theorem 2.3]{hat2007notes}.
    The theorem as stated there has five cases which we quote here using
    the notation conventions from those notes:
    \begin{enumerate}
        [label=(\alph*)]
        \item
        $M(0,1; \al/\bt),$ the various model Seifert fiberings of $S^1 \ti D^2.$
        \item
        $M(0,1; 1/2, 1/2) = M(-1,1;),$ two fiberings of
        \begin{math}
            S^1 \widetilde{\ti}
            S^1 \widetilde{\ti}
            I. \end{math}
        \item
        $M(0,0; \al_1/\bt_1, \al_2/\bt_2),$
        various fiberings of $S^3, S^1 \ti S^2$ and lens spaces.
        \item
        $M(0,0; 1/2, -1/2, \al/\bt) = M(-1, 0; \bt/\al)$ for $\al,\bt \ne 0.$
        \item
        $M(0,0; 1/2, 1/2, -1/2, -1/2) = M(-2, 0;),$
        two fiberings of 
        \begin{math}
            S^1 \widetilde{\ti}
            S^1 \widetilde{\ti}
            S^1.
        \end{math} \end{enumerate}
    Cases (a) and (b) are for manifolds with boundary
    and can be ignored here.
    We have used \cref{prop:UTinvt} above to rephrase case (e)
    in terms of unit tangent bundles.
    Since the manifold in case (e) is finitely covered by the 3-torus,
    one can show that
    it does not overlap with cases (c) or (d).
    Assuming for simplicity that $\al$ and $\bt$ are positive,
    the invariants in case (d) may be written in our notation as
    \[
        (0; (2, 1), (2, -1), (\al_3, \bt_3))
        \qandq
        (-1; (\al_1, \bt_1))
    \]        
    where $\al/\bt = \bt_1/\al_1 = \al_3/\bt_3$ and $\al_i > 0.$
    Note that if $\al_3 = 1,$ then 
    \[
        M(0; (2, 1), (2, -1), (1, \bt_3))
        =
        M(0; (2, -1), (2, 2 \bt_3 - 1))
    \]
    is a fibered marked lens space,
    and so there is actually some overlap
    between cases (c) and (d).
    Using the formula in \cite[\S6]{jn1983lectures},
    the fundamental group of the manifold
    $M(-1; (\al_1, \bt_1))$ is given by
    \[
        \langle a, q, h \ : 
        \ a \inv h a = h \inv,
        \ [h,q] = 1,
        \ q^{\al_1} h^{\bt_1} = 1,
        \ q a^2 = 1
        \rangle
    \]
    or equivalently
    \[
        \langle a, h \ : 
        \ a \inv h a = h \inv,
        \ a^{2 \al_1} = h^{\bt_1}
        \rangle.
    \]
    If $|\bt_1| > 1,$ this group is non-cyclic.
    and the manifold is not a lens space.
    (In fact, the manifold is a prism manifold.)
\end{proof}

Hence, the only fiberings on lens spaces left to study
are those of the form $(-1; (\al, \pm 1)).$
Up to a choice of orientation, these are all unit tangent bundles.

\begin{prop} \label{prop:otherfiber}
    Let $M$ be diffeomorphic to a lens space and suppose $M$ is Seifert
    fibered, but is not a fibered marked lens space.
    Then (up to a change of orientation) there is $\al \ge 1$
    such that $M$ is unit tangent bundle of the $\al \ti$ orbifold,
    $M$ is diffeomorphic to $L(4 \al, 2 \al \pm 1),$
    and the fibering has a \hvf{}.
\end{prop}
\begin{remark}
    Recall from the orbifold notation introduced in \cref{sec:orbifold}
    that the $\al \ti$ orbifold is $\bbR P^2$ with a cone point of order $\al$
    added. If $\al = 1,$ this orbifold is $\bbR P^2$ itself.
\end{remark}
\begin{proof}
    By \cref{prop:multifiber}, the fibering must have Seifert invariant
    $(-1; (\al, \pm 1)).$
    By \cref{cor:evchi}, up to a possible change of orientation
    we may assume that the invariant is $(-1; (\al, - 1)).$
    \Cref{prop:UTinvt} then implies that $M$ is the unit tangent
    bundle of the $\al \ti$ orbifold
    and \cref{thm:bigequiv} shows that it has a \hvf{}.
    By \cref{prop:multifiber}, this manifold is diffeomorphic
    to
    \[
        M(0; (2, 1), (2, -1), (1, \al))
        =
        M(0; (2, 1), (2, 2 \al - 1)).
    \]
    Using \cref{thm:lenspq} or directly from the computation
    \[
        \begin{pmatrix}
            2 & 2 \al - 1 \\
            1 & \al \end{pmatrix} 
        \begin{pmatrix}
            -1 & 0 \\
            0 & 1 \end{pmatrix}
        \begin{pmatrix}
            2 & 1 \\
            1 & 1 \end{pmatrix}
        \inv
        =
        \begin{pmatrix}
            -(2 \al + 1) & 4 \al \\
            * & 2 \al + 1
        \end{pmatrix}  \]
    one may show that this lens space is
    $L(4 \al, 2 \al + 1).$
    By \cref{thm:lenshomeo},
    $L(p, q)$ is diffeomorphic to $M$
    if and only if $|p| = 4 \al$ and $q \equiv 2 \al \pm 1 \bmod p$.
\end{proof}
This shows that beyond the fibered marked lens spaces,
there is exactly one additional Seifert fibering on
$L(4 \al, 2 \al \pm 1)$ for each $\al \ge 1.$
\Cref{thm:lens} now follows as a combination of
propositions \ref{prop:marklens} and \ref{prop:otherfiber}.

\section{Elliptic orbifolds} \label{sec:elliptic} %{{{1

We now consider Seifert fiberings over elliptic orbifolds.
We start by listing out all of the possible orbifolds.

\begin{prop} \label{prop:ellorbs}
    An elliptic orbifold is of one of the following forms:

\begin{multicols}{2}
    \begin{enumerate}
        \item the $pp$ orbifold with $p \ge 1;$
        \item
        the $2 2 p$ orbifold with $p \ge 2;$
        \item
        the $2 3 q$ orbifold with $3 \le q \le 5;$
        \item
        the $p \ti$ orbifold with $p \ge 1.$
    \end{enumerate}
\end{multicols}
\end{prop}
Note that the 2-sphere and real projective plane occur
as special cases with $p = 1.$
For further details and a proof of this result,
see the section of \cite{sco1983geometries} entitled
``\emph{The geometry of $S^2 \ti \bbR$}''
and the references therein.

\begin{prop} \label{prop:notlens}
    Let $M \to \Sig$ be a Seifert fibering over
    an elliptic base orbifold.
    Then $M$ is diffeomorphic to a lens space if and only if either
    \begin{enumerate}
        \item $\Sig$ is the $pp$ orbifold with $p \ge 1;$ or
        \item
        $M = \pm \UTSig$ where $\Sig$ is the \ $p \ti$ orbifold with $p \ge 1.$
    \end{enumerate} \end{prop}
\begin{proof}
    If $\Sig$ is the $pp$ orbifold, the Seifert invariant may be written
    as
    \[    
        (g; (\al_1, \bt_1), (\al_1, \bt_2))
    \]
    with $\al_1 = \al_2 = p$
    and so $M$ is a fibered marked lens space.
    If $\Sig$ is the $p \ti$ orbifold, then
    propositions \ref{prop:multifiber} and \ref{prop:otherfiber}
    imply that $M$ is diffeomorphic to a lens space if and only if
    $M = \pm \UTSig.$
    The other elliptic orbifolds listed in \cref{prop:ellorbs}
    may be ruled out by \cref{prop:multifiber}.
\end{proof}

Now, consider Seifert fiber spaces $M_1$ and $M_2$ over an orbifold $\Sig$
and suppose there is a fiberwise covering $M_1 \to M_2$ of degree $d > 1.$
\Cref{cor:eulerd} states that $d \cdot e(M_1 \to \Sig) = e(M_2 \to \Sig).$
Roughly speaking,
the Euler number measures the amount of ``twisting'' of the
fibering, and the effect of quotienting $M_1$ down to $M_2$
is to increase this twisting by a factor of $d.$
Conversely, if we consider $M_2$ first,
finding a fiberwise covering map $M_1 \to M_2$ 
is in some sense equivalent to finding a way
to reduce the twisting of the fibering by a factor of $d.$
If we know, however, that the fibering $M_2 \to \Sig$
already has the minimum possible amount of (non-zero)
twisting for a fibering over $\Sig,$ then it cannot be fiberwise
covered by another fibering.

This is the technique we will use to prove \cref{thm:elliptic}.
We will show for all but one of the cases
that $e(\UTSig \to \Sig)$ is the smallest possible positive Euler number
for any Seifert fibering over $\Sig.$
In the exceptional case, a degree two covering is possible,
but $M$ must be a lens space.

\begin{prop} \label{prop:ellcover}
    Suppose $\Sig$ is an elliptic orbifold and
    $M \to \UTSig$ is a fiberwise covering
    of positive degree $d.$
    Then the degree $d$ is at most two.
    Moreover, the $d = 2$ case occurs exactly when $M$ is of the form
    \begin{math}
        M(0; \ (\al, \bt_1), \ (\al, \bt_2))
    \end{math}
    where $\al \ge 1$ is odd,
    \[    
        \bt_1 \ = \ \tfrac{\al}{2} \ - \ \tfrac{1}{2}
        \qandq
        \bt_2 \ = \ -\tfrac{\al}{2} \ - \ \tfrac{1}{2}.
    \] \end{prop}
The two-fold covering above with $\al = 1$
corresponds to the so-called ``belt trick'' or ``plate trick''
where $S^3$ double-covers $\UT S^2.$

\begin{proof}
    We proceed through the cases in \cref{prop:ellorbs}
    in an order which makes the overall proof easiest to follow.

    %smallskip()
    
    \pagebreak[3]

    \textbf{Case $1.$} $\Sig$ is the $23q$ orbifold for $3 \le q \le 5.$
    % -1 + 1/2 + 1/3 + 1/5 =
    % -1/6 + 1/5 =
    % -5/30 + 6/30 =
    % (6-5)/30 = 
    % 1/30

    We give a full proof for the $235$ orbifold.
    Here $e(\UTSig \to \Sig) = \chi(\Sig) = 1/30$
    and $M$ has a Seifert invariant which may be written as
    \begin{math}
        (0; (2, \bt_1), (3, \bt_2), (5, \bt_3))
    \end{math}
    for some values of $\bt_i.$
    The Euler number of $M \to \Sig$ is a linear combination
    of $1/2$, $1/3$ and $1/5$ with integer coefficients
    and hence must be an integer multiple of 1/30.
    Say $k$ is such that $e(M \to \Sig) = k/30.$
    Then $d \cdot e(M \to \Sig) = e(\UTSig \to \Sig)$
    implies that $d k = 1.$
    As $d$ and $k$ are integers and $d$ is positive,
    the only possibility is that $d = k = 1.$
    The $233$ and $234$ orbifolds have similar proofs
    with 1/6 and 1/12 replacing 1/30 above.
    % For 233 orbifold:
    % -1 + 1/2 + 1/3 + 1/3 =
    % -1/6 + 1/3 =
    % 1/6
    % For 234 orbifold
    % -1 + 1/2 + 1/3 + 1/4 =
    % -1/6 + 1/4 =
    % -2/12 + 3/12 =
    % 1/12

    \smallskip{}

    \textbf{Case $2.$} $\Sig$ is the $22p$ orbifold where $p$ is even.

    Here $e(\UTSig \to \Sig) = \chi(\Sig) = 1/p$
    and the Euler number of $M \to \Sig$ is a linear combination
    of $1/2$ and $1/p$ with integer coefficients.
    Hence, $e(M \to \Sig) = k/p$ for some integer $k$
    and $d k = 1$ implies that $d = k = 1.$

    \smallskip{}

    \textbf{Case $3.$} $\Sig$ is the $22p$ orbifold where $p$ is odd.

    As in the previous case,
    $e(\UTSig \to \Sig) = \chi(\Sig) = 1/p$
    and the Euler number of $M \to \Sig$ is a linear combination
    of $1/2$ and $1/p$ with integer coefficients.
    However, as $p$ is odd, we may only conclude that
    $e(M \to \Sig) = k/2p$ for some integer $k$
    and so $d k = 2.$
    As shown by \cref{prop:poscoverd},
    when quotienting along fibers to produce a $d$-fold
    cover, the integer $d$ must be coprime to the orders
    of all of the cone points.
    As $\Sig$ has cone points of order $2, d$ cannot be even
    and so $d = 1$ is the only possibility.

    \smallskip{}

    \textbf{Case $4.$} $\Sig$ is the $p \ti$ orbifold.

    Here $e(\UTSig \to \Sig) = \chi(\Sig) = 1/p$
    and $M$ has a Seifert invariant which may be written as
    \begin{math}
        (-1; (p, \bt_1))
    \end{math}
    for some value of $\bt_1.$
    Hence, $e(M \to \Sig) = k/p$ for some integer $k$
    and $d k = 1$ implies that $d = k = 1.$

    \smallskip{}

    \textbf{Case $5.$} $\Sig$ is the $pp$ orbifold where $p$ is even.

    Here $e(\UTSig \to \Sig) = \chi(\Sig) = 2/p$
    and $e(M \to \Sig) = k/p$ for some integer $k.$
    Hence, $d k = 2.$
    By \cref{prop:poscoverd},
    $d$ is coprime to $p$ and so must be odd.
    This is only possible if $d = 1.$

    \smallskip{}

    \textbf{Case $6.$} $\Sig$ is the $pp$ orbifold where $p$ is odd.

    As in the previous case,
    $d k = 2$ where $d$ is coprime to $p.$
    As $p$ is odd, both $d = 1$ and $d = 2$ are possible.
    Assume $d = 2$ for the rest of the proof.
    Case $(4)$ of \cref{thm:bigequiv} holds with $d = 2$ and $\al = \al_1 = \al_2 = p,$
    showing that
    \[
        d \cdot e(M \to \Sig)
        = - \tfrac{2}{\al} (\bt_1 + \bt_2)
        = \tfrac{2}{\al}
        = \chi(\Sig)
        \qandq
        2 \bt_i / \al \equiv -1 / \al \bmod \bbZ.
    \]
    As $\al$ is odd, it follows that
    \begin{math}
        {\bt_i} \equiv \tfrac{1}{2} (\al - 1) \bmod \al.
    \end{math}
    Using the allowed manipulations on Seifert invariants,
    we may adjust the ratios $\bt_i/\al_i$ by integer amounts
    and assume with loss of generality
    that
    \begin{math} 
        \bt_1 \ = \ \tfrac{\al}{2} \ - \ \tfrac{1}{2}.
    \end{math}        
    Then $\bt_1 + \bt_2 = -1$ implies that
    \begin{math}
        \bt_2 \ = \ -\tfrac{\al}{2} \ - \ \tfrac{1}{2}.
        \qedhere
    \end{math} \end{proof}
Proving \cref{thm:elliptic} is now a matter of combining the above results.

\begin{proof}
    [Proof of \cref{thm:elliptic}]
    By \cref{thm:bigequiv}, having a \hvf{} is equivalent to having
    a fiberwise covering $M \to \UTSig$ of some degree $d.$
    Up to changing the orientation of $M,$
    we may assume $d \ge 1.$
    Since we are assuming $M$ is not diffeomorphic to a lens space,
    propositions \ref{prop:notlens} and \ref{prop:ellcover}
    imply that $d = 1$ and so $M = \UTSig.$
\end{proof}
As a final observation,
note that for each $p > 1,$ the unit tangent bundle of the $22p$ orbifold
has two fiberings
\[
    M(0; (2,1), (2,-1), (p, -1)
    \qandq
    M(-1; (1, -p))
\]
given by item $(2)$ of \cref{prop:multifiber}.
The first of these fibering is the standard fibering of the unit tangent
bundle, and so has a \hvf{}.
The second is a fibering over the real projective plane
and since $p > \chi(\bbR P^2) = 1,$ this fibering does not have a \hvf{}.

\section{Parabolic orbifolds} \label{sec:para} %{{{1

We now consider Seifert fiberings over elliptic orbifolds.
There are exactly seven orbifolds of this type.

\begin{prop} \label{prop:parorbs}
    The parabolic orbifolds are as follows:

\begin{multicols}{2}
    \begin{enumerate}
        \item the 2-torus $\bbT^2;$
        \item
        the Klein bottle $K;$
        \item
        the $236$ orbifold;
        \item
        the $244$ orbifold;
        \item
        the $333$ orbifold;
        \item
        the $2222$ orbifold;
        \item
        the $22 \ti$ orbifold.
    \end{enumerate}
\end{multicols}
\end{prop}
For further details and a proof of this result,
see the section of \cite{sco1983geometries} entitled
``\emph{The geometry of $E^3$}''
and the references therein.
See also \S2.1 of \cite{hat2007notes}.
It is actually a somewhat enjoyable exercise to prove the result directly
by starting with a surface $\Sig_0$ with $\chi(\Sig_0) \ge 0$
and seeing which combinations of cone points may be added to get
$\chi(\Sig)$ down to exactly zero.

\subsection{Self-coverings}
For each of the parabolic orbifolds,
$\UTSig$ is an oriented Seifert fiber space
and \cref{cor:evchi} implies that $e(\UTSig \to \Sig) = \chi(\Sig) = 0.$
It turns out that, up to orientation,
these unit tangent bundles
are the only Seifert fiber spaces of this form.

\begin{prop} \label{prop:zerozero}
    There are exactly ten distinct Seifert fiberings $M \to \Sig$
    for which $M$ is oriented and $e(M \to \Sig) = \chi(\Sig) = 0.$
    These are all of the form $M = \pm \UTSig.$
    Moreover, if $\Sig$ is the $236$, $244$, or $333$ orbifold,
    then $\UTSig \ne -\UTSig.$
    For the other parabolic orbifolds, $\UTSig = -\UTSig.$
\end{prop}
\begin{proof}
    A list of these Seifert invariants appears in a number of places.
    For instance, \S8.2 of \cite{orl1972seifert}
    and the sections of \cite{sco1983geometries} and \cite{hat2007notes}
    mentioned above.
    As the result is not difficult,
    we outline the proofs for two of the cases here and leave the others,
    which are similar, to the reader.

    \textbf{The $236$ orbifold.} \
    Consider a Seifert fibering $M \to \Sig$
    having invariant
    \[    
        (0; (2, \bt_1), (3, \bt_2), (6, \bt_3))
    \]
    and
    $e(M \to \Sig) = 0.$
    As $\bt_2$ is coprime to $3,$
    we have either $\bt_2 \equiv 1$ or $\bt_2 \equiv -1 \bmod 3.$
    By reversing the orientation of the manifold if necessary
    and using \cref{prop:orientd},
    we may assume the latter holds.
    By doing valid operations on the Seifert invariant,
    adding an integer to one ratio $\bt_i/\al_i$ and subtracting it to another
    we may reduce to the case where the invariant
    is ($0; (2, -1), (3, -1), (6, \bt_3).$
    Then $e(M \to \Sig) = 0$ implies that $\bt_3 = 5$ and one may verify
    $M = \UTSig.$

    \textbf{The $333$ orbifold.} \
    Consider a Seifert fibering $M \to \Sig$
    having invariant
    \[    
        (0; (3, \bt_1), (3, \bt_2), (3, \bt_3))
    \]
    and
    $e(M \to \Sig) = 0.$
    Each $\bt_i$ is coprime to $3.$
    One can show that $\bt_1 \equiv \bt_2 \equiv \bt_3 \bmod 3$
    as otherwise $e(M \to \Sig)$ would not be an integer.
    Applying \cref{prop:orientd} if necessary,
    we may assume $\bt_1 \equiv \bt_2 \equiv \bt_3 \equiv -1.$
    Similar to above,
    we may reduce to the case where the invariant
    is ($0; (3, -1), (3, -1), (3, \bt_3).$
    Then $e(M \to \Sig) = 0$ implies that $\bt_3 = 2$ and
    $M = \UTSig.$

    Using \cref{prop:orientd}, one can determine whether or not
    $\UTSig$ and $-\UTSig$ have the same Seifert invariants.
\end{proof}    %Keeping tract of the duplicate cases UTSig = -UTSig
    %one sees that there are exactly 2 cdot 7 - 4 = 10 fiberings.

\begin{prop} \label{prop:paracover}
    Suppose $\Sig$ is an elliptic orbifold
    with a (possibly empty) set of cone points
    of orders $\al_1, \ldots, \al_n.$
    If $M \to \UTSig$ is a fiberwise cover, then $M = \pm \UTSig.$
    Moreover, such a fiberwise cover of degree $d$
    exists if and only if $d$ is coprime to all of the $\al_i.$
\end{prop}    
\begin{proof}
    First, suppose $M_1 \to \UTSig$ is a fiberwise covering.
    Then, \cref{cor:eulerd} implies that 
    $d \cdot e(M_1 \to \Sig) = \chi(\Sig) = 0$
    and so
    $e(M_1 \to \Sig) = 0.$
    As $M_1$ has the same base orbifold as $\UTSig,$
    \cref{prop:zerozero} applies and $M_1 = \pm \UTSig.$

    Conversely, starting from an integer $d$ coprime to all $\al_i,$
    \cref{prop:coverd} gives a fiberwise covering $\UTSig \to M_2.$
    \Cref{cor:eulerd} implies that
    %e(M_2 -> Sig) = d cdot e(UTSig -> Sig) = 0
    $e(M_2 \to \Sig) = 0$
    and then \cref{prop:zerozero} implies that $M_2 = \pm \UTSig.$
    Hence, we have constructed a covering $\UTSig \to \pm \UTSig.$
    Up to a possible change of orientation,
    this gives a covering $\pm \UTSig \to \UTSig$ as desired.
\end{proof}
Using the above propositions, it is a straightforward task to list
out all of the possible coverings:
\begin{enumerate}
    \item If $\Sig = \bbT^2,$ then \\
    $\UTSig \cong -\UTSig$ and
    has a $d$-fold self-covering for all $d \ne 0.$
    \item
    If $\Sig = K$ is the Klein bottle, then \\
    $\UTSig \cong -\UTSig$ and
    has a $d$-fold self-covering for all $d \ne 0.$
    \item
    \covercase{236}{6}
    \item
    \covercase{244}{4}
    \item
    \covercase{333}{3}
    \item
    If $\Sig$ is the $2222$ orbifold, then \\
    $\UTSig \cong -\UTSig$ and
    $d$-fold covers itself if
    $d \equiv 1$ mod $2.$
    \item
    If $\Sig$ is the $22\ti$ orbifold, then \\
    $\UTSig \cong -\UTSig$ and
    $d$-fold covers itself if
    $d \equiv 1$ mod $2.$
\end{enumerate}
\medskip{}

We now attempt to impart some intuition
about the structure of these self-coverings and their associated \hvfs{}.
Every parabolic orbifold is the quotient of $\bbR^2$
by a wallpaper group
consisting of affine isometries.
Consider the unit tangent bundle $\UTbbR^2$
of the Euclidean plane.
An element of
\begin{math}
    \UTbbR^2 \cong \bbR^2 \ti \bbS^1
\end{math}
may be represented by a pair $(x, \theta)$ where
$x \in \bbR^2$ and $\theta$ is an angle.
Then $\UTbbR^2$ has self-covering maps of the form
\begin{math}
    (x, \theta) \mapsto (x, d \theta)
\end{math}
for $d \ge 1.$

If $A : \bbR^2 \to \bbR^2$ is an orientation-preserving isometry,
then its derivative is of the form
\begin{math}
    (x, \theta) \mapsto (A(x), \, \theta + \theta_0)
\end{math}
for some constant $\theta_0.$
In some cases, this will commute with the $d$-fold self-covering;
for instance, when $A$ is a rotation by an angle $\theta_0 = \frac{2\pi}{k}$
and $d \equiv 1$ mod $k.$
If the self-covering commutes with every element of the wallpaper group,
then this defines a self-covering map on unit tangent bundle of the orbifold.

As a concrete example, consider the wallpaper group associated
to the $236$ orbifold.
This group is generated by three rotations of orders $2, 3,$ and $6$
as depicted in \cref{fig:wallpaper}.
\begin{figure}
    \centering
    \includegraphics{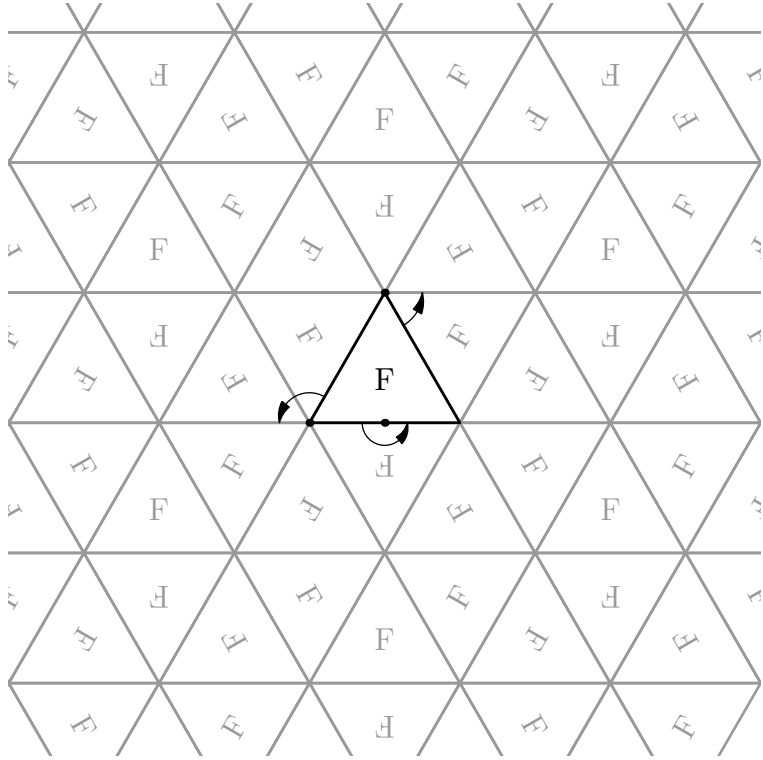}
    \caption[Wallpaper]{ Generators of the wallpaper group associated
    to the $236$ orbifold.}
    \label{fig:wallpaper}
\end{figure}
If $d \equiv 1 \bmod 6,$ the self-covering on $\UT \bbR^2$ will quotient
to the unit tangent bundle $\UTSig.$

\subsection{Vector fields on surfaces}
To completely handle Seifert fiberings over parabolic orbifolds,
we must also consider case $(1)$ of \cref{prop:seicover}.

\begin{prop} \label{prop:factor}
    Let $\Sig$ be the 2-torus or Klein bottle.
    Then any Seifert fibering $M \to \Sig$
    has a \hvf{} (whether or not $M$ is orientable).
\end{prop}
\begin{proof}
    On such a surface, one can construct
    a unit vector field $v : \Sig \to \UTSig.$
    Define $u : \Sig \to \UTSig$ as the composition of $v$ with
    the projection $\pi : M \to \Sig$ defining the Seifert fibering.
    Then, a \hvf{} exists by \cref{prop:hvfmap}.
\end{proof}
For each of $\bbT^2$ and $K,$
there are infinitely many oriented Seifert fiberings,
one for each integer $e(M \to \Sig).$
These manifolds have Euclidean geometry if $e(M \to \Sig) = 0$
and Nil geometry otherwise
\cite{sco1983geometries}.
There are exactly four non-orientable 3-manifolds with Euclidean
geometry.
(See \S8.2 of \cite{orl1972seifert}.)
Each of these has a Seifert fibering over $K$
and therefore has a \hvf{}.
Two of these manifolds also have Seifert fiberings over $\bbT^2$
which therefore also have \hvfs{}.

%TODO:
%    Decide if I want to give readers the gory details
%    in the non-orientable case.

\begin{proof}
    [Proof of \cref{thm:parabolic}]
    If $\Sig$ is a surface with $\chi(\Sig) = 0,$
    then \cref{prop:factor} establishes the theorem,
    so suppose $\Sig$ is not a surface.
    \Cref{thm:bigequiv} and \cref{prop:paracover} then
    imply that
    \[  \text{$M$ has a \hvf{}}
        \  \Leftrightarrow \  
        \text{$M$ fiberwise covers $\UTSig$}
        \  \Leftrightarrow \  
        M = \pm \UTSig.
        \qedhere
    \] \end{proof}

\section{Hyperbolic Orbifolds} \label{sec:hyp} %{{{1

The final class of orbifolds to consider are those with hyperbolic geometry.
Even though this is the most general case,
there is the least to say here.
In this setting,
it is relatively easy to find fiberwise coverings $M \to \UTSig$ of degree $d > 1.$
For circle bundles,
if $\Sig$ is a surface of genus $g \ge 2,$ then there is a cover of degree $d$
for every factor $d$ of $\chi(g) = 2 - 2g.$
Even for orbifolds without handles, non-trivial covers are possible.
For instance, using \cref{thm:bigequiv}, one may show that
\[
    M(0; (1,-1), (5,2), (5,2), (5,2))
\]
double covers the unit tangent bundle of the $555$ orbifold.

For certain choices of cone points,
there may be no non-trivial covers.
For instance, if $\Sig$ is the $237$ orbifold,
one may show that only $d = \pm 1$ is possible.
The proof is of the same form as Case $1$ of the proof of \cref{prop:ellcover}.

% Here e(M->Sig) is -(-1+ 2/5 + 2/5 + 2/5) = -1/5 
% and chi(Sig) is 2 - 3*(4/5) = (10 - 12)/5 = -2/5

If $\Sig$ is a hyperbolic orbifold,
the geodesic flow on $\UTSig$ is an Anosov flow
and the flow is generated by a vector field which is horizontal.
This flow lifts to any finite cover and is still Anosov on the cover.
Ghys and Barbot showed that (up to orbit equivalency)
every Anosov flow on a Seifert fiber space is of this form
\cite{ghy1984flots,bar1996flots}.
This establishes 
the equivalence $(2)$ $\Leftrightarrow$ $(3)$ in \cref{thm:hyperfull}.
The equivalence $(1)$ $\Leftrightarrow$ $(2)$
is a re-statement of the results in \cref{sec:general}.

%section(Model geometries) %{{{1
%
%As mentioned in the introduction,
%the condition of having a hvf()
%may be either more or less restrictive than
%having a horizontal foliation
%depending on the model geometry of the Seifert fiber space.
%We briefly go through the cases here.
%These statements follow from the results of the preceeding sections
%and from previous work on horizontal foliations
%cite(ehn1981transverse, nai1994foliations, sco1983geometries).
%For this list, we only consider orientable 3-manifolds.
%
%enum:
%    For the geometry of bbS^3,
%    there are many fiberings with hvfs(),
%    but none have horizontal foliations.
%    ---
%    For the geometry of bbS^2 ti bbR,
%    there are no hvfs(),
%    but all fiberings have horizontal foliations with compact leaves.
%    ---
%    For Euclidean geometry,
%    all fiberings have both hvfs() and
%    horizontal foliations with compact leaves.
%    ---
%    For Nil geometry,
%    if there are no exceptional fibers
%    then the fibering has a hvf(),
%    but no fiberings have a horizontal foliation.
%    ---
%    For the geometry of bbH^2 ti bbR,
%    there are no hvfs(),
%    but all fiberings have horizontal foliations with compact leaves.
%    ---
%    For the geometry of SL(2,bbR),
%    any fibering which has a hvf()
%    also has horizontal foliation coming
%    from the weak stable foliation of the Anosov flow.
%    There are also fiberings which have horizontal foliations,
%    but do not have hvfs().

\section{Homotopies of \hvfs{}} \label{sec:homotopy} %{{{1

For a given Seifert fibering $M \to \Sig,$
consider the space of all \hvfs{}.
What are the connected components of this space?
That is, when can one vector field be deformed into another
along a path of \hvfs{}?
For simplicity, we only consider this in the case where
the base orbifold $\Sig$ is oriented.
%though similar results hold in the non-orientable case.
By \cref{prop:hvfmap}, the question reduces
to studying homotopy classes of maps of the form
\[
    \simpleCD{M}{u}{\UTSig}{}{}{}{\Sig}{\id}{\Sig.}
\]
\smallskip{}

Consider two such maps $u, v : M \to \UTSig$
and suppose they have the same degree $d.$
(Otherwise, they are clearly not homotopic.)
By the averaging method used in the proof of \cref{prop:seicover},
we may assume there are metrics on the fibers of 
$M$ and $\UTSig$ and that $u$ and $v$ have the same constant speed on all fibers.
As $\Sig$ is oriented, the fibers of $\UTSig$ are oriented.
Using the metric, there is then a well-defined difference map
\begin{math}
    u - v : M \to \bbS^1
\end{math}
measuring the angle between vectors.
This map is constant on fibers
and therefore quotients to a map $g : \Sig \to \bbS^1.$
The map $u$ is homotopic to $v$
if and only if
$g$ is homotopic to a constant map.
Since we are only concerned with the homotopy class of $g,$
the smooth orbifold structure of $\Sig$ is unimportant
and we may consider $g$ as a map from $\Sig_0$ to $\bbS^1$
where $\Sig_0$ is the underlying topological surface.

Note that these steps are reversible.
Given a continuous map $g : \Sig_0 \to \bbS^1,$
we may, up to homotopy, assume $g$ is smooth everywhere
and constant on a neighbourhood of the cone points.
If $\pi : M \to \Sig$ is the Seifert fibering,
then $g \circ \pi : M \to \bbS^1$ is a smooth map.
Given $u$ as above, we may construct $v$ by $v = u + g \circ \pi$
where the plus sign denotes addition of angles.
If we assume $u$ is fixed, then this gives a bijection
between homotopy classes of maps $v : M \to \UTSig$
of the same degree as $u$ and homotopy classes
of maps $g : \Sig_0 \to \bbS^1.$

There are canonical isomorphisms identifying
\begin{enumerate}
    \item the homotopy classes of maps from $\Sig_0$ to $\bbS^1,$
    \item
    homomorphisms from $\pi_1(\Sig_0)$ to $\bbZ = \pi_1(\bbS^1),$
    \item
    homomorphisms from the first homology group $H_1(\Sig_0)$ to $\bbZ,$ and
    \item
    elements of the first cohomology group $H^1(\Sig_0,\bbZ).$
\end{enumerate}
For $\Sig_0 = \bbS^2,$ this follows because all of the above are trivial.
For other oriented surfaces, it follows because both $\Sig_0$ and $\bbS^1$
are $K(\pi,1).$
(See also Theorem 4.57 and the discussion on page $198$ of \cite{hat2002algebraic}.)
% TODO: More explanation?

Each connected component of the space of \hvfs{}
is therefore uniquely determined by the degree $d$
(which is also the number of turns that the vector field
makes around any fiber)
and an element of $H^1(\Sig_0,\bbZ).$

For a Seifert fibering $M \to \Sig$ of an oriented manifold,
let $\DMSig$ denote the ``allowable degrees''
of a \hvf{}.
That is $d \in \DMSig$ if and only if
there is a degree-$d$ map $u : M \to \UTSig$ of the form given
in \cref{prop:hvfmap}.
The above reasoning may then be used to show the following.

\begin{thm} \label{thm:hvfcomps}
    For a Seifert fibering $M \to \Sig$ over an oriented orbifold,
    there is a bijection between the connected components
    of the space of \hvfs{} and pairs of the form
    \begin{math}
        (d, \varphi) \in \DMSig \ti H^1(\Sig_0,\bbZ).
    \end{math} \end{thm}
    
If $\chi(\Sig) \ne 0,$ then \cref{thm:bigequiv} shows that
the degree $d$ is uniquely determined by the fibering.
In particular,
for a bad, elliptic, or hyperbolic orbifold with $\Sig_0 = \bbS^2,$
\cref{thm:hvfcomps} implies that
the Seifert fibering uniquely determines the \hvf{} up to homotopy.
Conversely if $H^1(\Sig_0,\bbZ)$ is non-trivial,
then many homotopy classes are possible.
For instance,
consider the 3-torus $\bbT^3$ with fibers tangent to
the vertical $z$-direction.
Then $U T \bbT^2$ may be identified with $\bbT^3$ and
the classes of \hvfs{} correspond to
%(via cref(prop:hvfmap))
maps of the form
\[
    \bbT^3 \to \bbT^3, \quad (x,y,z) \mapsto (x,y, a x + b y + d z)
\]
for all $(a,b,d) \in \bbZ^3.$

%Note that we did not use any properties specific to unit tangent
%bundles.
%The same reasoning shows that for two given Seifert fiberings
%with oriented fibers, the homotopy class of
%fiber-preserving maps of the form
%
%    \simpleCD{M_1}{F}{M_2}{}{}{}{Sig}{f}{Sig.}
%
%are determined by the degree of F on fibers
%and and element of H^1(Sig_0,bbZ).

\section{Manifolds with boundary} \label{sec:boundary} %{{{1

In this final section, we consider Seifert fiberings on
manifolds with boundary.
First, consider the case where there are boundary conditions
for the vector field.
We could require that the vector field be either tangent or transverse
to the boundary and the resulting restrictions on $M$ will be the same.

\begin{thm} \label{thm:tanbound}
    For a Seifert fibering $M \to \Sig$ on a manifold with boundary,
    the following are equivalent:
    \begin{enumerate}
        \item $M$ has a \hvf{} everywhere tangent to the boundary,
        \item
        $M$ has a \hvf{} everywhere transverse to the boundary,
        \item
        the base orbifold $\Sig$ is either the annulus or the M\"obius band.
    \end{enumerate} \end{thm}
\begin{proof}
    We show $(2)$ $\Leftrightarrow$ $(3).$
    The proof of $(1)$ $\Leftrightarrow$ $(3)$
    is similar and left to the reader.
    Note that \cref{prop:seicover} holds
    for manifolds with boundary using the same proof.
    Suppose $M$ supports a \hvf{} which is everywhere transverse
    to $\del M$ and consider the associated map $u : M \to \UTSig.$
    For a point $x \in \del \Sig,$
    the restriction of $u$ to the fiber over $x$
    cannot be surjective as its range omits
    the two unit vectors at $x$ which are tangent to $\del \Sig.$
    This implies that $u$ has degree zero
    and \cref{prop:seicover} shows that $u$ is homotopic
    to a composition $M \to \Sig \to \UTSig.$
    In the proof of the proposition,
    the vector field $\Sig \to \UTSig$ is
    defined by averaging
    and one can verify that is it transverse to $\del \Sig.$
    Further, the same argument as in \cref{prop:orbvfield}
    shows that $\Sig$ has no cone points and is therefore
    a surface with boundary.
    The Poincar\'e--Hopf theorem then implies that
    $\Sig$ is either the annulus or M\"obius band.

    Conversely, for any circle bundle
    over an annulus or M\"obius band,
    we may compose the projection $M \to \Sig$
    with a vector field $\Sig \to \UTSig$ transverse to $\del \Sig$
    and produce a \hvf{} on $M$ transverse to $\del M.$
\end{proof}
For the remainder of the section,
we consider \hvfs{} with no boundary conditions.

\begin{thm} \label{thm:nobound}
    Let $M \to \Sig$ be a Seifert fibering on a manifold with boundary.
    Then $M$ has a \hvf{} if and only if
    one or both of the following hold:
    \begin{enumerate}
        \item the base orbifold $\Sig$ is a surface with boundary, or
        \item
        $M$ fiberwise covers the unit tangent bundle of $\Sig.$
    \end{enumerate} \end{thm}
To prove this, we use the following.

\begin{lemma} \label{lemma:keynobound}
    An orbifold with boundary supports a non-zero vector field
    if and only if it has no cone points
    (i.e., it is a surface with boundary).
\end{lemma}
\begin{proof}
    As shown in the proof of \cref{prop:orbvfield},
    it is impossible to define a non-zero vector field in the neighbourhood
    of a cone point
    and so the existence of such a vector field implies that $\Sig$ is a surface with boundary.
    Conversely,
    if we have no boundary conditions,
    we may construct a non-zero vector field on any surface with boundary.
    For instance,
    we may take a generic vector field on a closed surface.
    This is zero at finitely many points \cite{pei1962structural},
    and we can excise one or more topological disks to remove all of these
    points.
\end{proof}
Using this lemma and checking that the results of \cref{sec:general}
extend to the case of a manifold with boundary,
one may then prove \cref{thm:nobound}.

For an oriented manifold with boundary, write the Seifert invariant as
\[
    (g, n; (\al_1, \bt_1), \ldots, (\al_k, \bt_k))
\]
where $n > 0$ is the number of boundary components.
We can reorder the pairs, add or remove pairs of the form $(1,0)$,
and (specifically for $\del M \ne \varnothing$) add an integer to
any of the ratios $\bt_i/\al_i$ without changing the fibering.
Note here that integer changes to one ratio $\bt_i/\al_i$ do
not need to be offset with a change to another ratio  $\bt_j/\al_j.$
Because of this, the Euler number $e(M \to \Sig)$ is not well defined.
See \S2 of \cite{hat2007notes} for more details.

By adapting arguments in \cref{sec:general}
one can prove the following analogue of \cref{thm:bigequiv}.
Note that the condition
\begin{math}
    d \cdot e(M \to \Sig) = \chi(\Sig)
\end{math}
has been removed.

\begin{prop} \label{prop:boundaryequiv}
    Let $M$ be a manifold with boundary 
    and let $M \to \Sig$ be a Seifert fibering
    such that $\Sig$ has one or more cone points.
    Then the following are equivalent:
    \begin{enumerate}
        \item $M$ has a \hvf{};
        \item
        there is a continuous map $u: M \to \UTSig$
        such that
        \[
            \simpleCD{M}{u}{\UTSig}{}{}{}{\Sig}{\id}{\Sig}
        \]
        commutes;
        \item
        there is a fiberwise covering from $M$ to
        the unit tangent bundle $\UTSig;$
        \item
        $M$ is orientable, the fibering has Seifert invariant
        \[
            (g, n; (\al_1, \bt_1), \ldots, (\al_k, \bt_k))
        \]
        and there is a non-zero integer $d$
        such that
        \[        
            d \bt_i/\al_i \equiv -1/\al_i    \mod \bbZ
        \]
        for all $i = 1, \ldots, k.$
    \end{enumerate} \end{prop}

Even with the Euler number condition removed,
it may not be possible to find a \hvf{}.
Consider, for instance $(g,n; (3,1), (3,2))$.
No integer $d$ satisfies
\begin{math}
    \tfrac{1}{3} d \equiv \tfrac{2}{3} d \equiv \tfrac{-1}{3}    \bmod \bbZ
\end{math}
and so no \hvf{} exists.
%A similar argument holds for $(g,n; (6, 1), (15, 2))$.

If a Seifert fibering satisfies item $(4)$ of \cref{prop:boundaryequiv}
for some integer $d,$
then it also satisfies the condition when
$d$ is replaced by $d + m \ell$
where $m$ is any integer and
$\ell$ is the least common multiple of $\{\al_1, \ldots, \al_k\}.$
Using this and adapting the results in \cref{sec:homotopy},
one may show that if a Seifert fibering on a manifold with boundary
has one \hvf{}, then it has infinitely many homotopy classes
of such vector fields.

\bigskip

\noindent\textbf{Acknowledgements.}
The author thanks
Jonathan Bowden,
Heiko Dietrich,
Rafael Potrie,
Jessica Purcell,
Mario Shannon, and
Stephan Tillmann
for helpful conversations.
We also thank the reviewer for helpful comments.
This research was partially funded by the Australian Research Council.

% Epilogue {{{1

% Hack to fit to one less page
%vspace(-5mm)

\bibliographystyle{alpha}
\bibliography{dynamics}

\end{document}